\documentclass[11pt,oneside,reqno]{amsart}
\usepackage{amsmath,amsfonts}
\usepackage{esint}

\usepackage{tikz}
\usepackage{dsfont} 
\usepackage{stmaryrd,mathrsfs,bm,amsthm,mathtools,yfonts,amssymb,color}
\usepackage[a4paper,left=30mm,right=30mm,top=35mm,bottom=35mm,marginpar=25mm]{geometry}

\usepackage{xfrac}
\usepackage{xcolor}
\usetikzlibrary{positioning}
\usepackage{mathtools}

\usepackage[shortlabels]{enumitem}

\usepackage{hyperref} 
\usepackage[capitalise, nameinlink, noabbrev]{cleveref} 
\hypersetup{colorlinks=true, 
	linkcolor=blue,
	citecolor=red,
}

\usepackage{tikz}
\usepackage{sansmath}
\usetikzlibrary{shadings,intersections,patterns}

\usepackage{pgfplots}

\usepgfplotslibrary{fillbetween}

\theoremstyle{plain} 

\newtheorem{theorem}{Theorem}[section]

\newtheorem{corollary}[theorem]{Corollary}
\newtheorem{lemma}[theorem]{Lemma}
\newtheorem{definition}[theorem]{Definition}
\theoremstyle{oss}
\newtheorem{oss}[theorem]{Remark}

\newcommand{\C}{\mathbb{C}}
\newcommand{\R}{\mathbb{R}}

\newcommand{\N}{\mathbb{N}}

\newcommand{\eps}{\varepsilon}

\makeindex

\def\XXint#1#2#3{{\setbox0=\hbox{$#1{#2#3}{\int}$}
		\vcenter{\hbox{$#2#3$}}\kern-.5\wd0}}

\numberwithin{equation}{section}

\title[(Quasi-)conformal methods in free boundary problems]{(Quasi-)conformal methods in two-dimensional free boundary problems}

\author[G.~De Philippis]{Guido De Philippis}
\address{\textit{G.~De Philippis:} Courant Institute of Mathematical Sciences, New York University, 251 Mercer St., New York, NY 10012, USA.}
\email{guido@cims.nyu.edu}

\author[L.~Spolaor]{Luca Spolaor}
\address{\textit{L.~Spolaor:} Department of Mathematics, University of California, San Diego, La Jolla, CA, 92093
}
\email{lspolaor@ucsd.edu}

\author[B.~Velichkov]{Bozhidar Velichkov}
\address{\textit{B.~Velichkov:} 	
	Dipartimento di Matematica, Universit\`a di Pisa, 
	Largo Bruno Pontecorvo, 5, 56127 Pisa - ITALY}
\email{bozhidar.velichkov@unipi.it}

\thanks{{\bf Acknowledgments.} 
	G.D.P has been partially supported by the NSF grant DMS 2055686 and by the Simons Foundation.  L.S. has been partially supported by the NSF grant DMS 1810645 and by the N.S.F Career Grant DMS 2044954.  B.V. was supported by the European Research Council (ERC), under the European Union's Horizon 2020 research and innovation programme, through the project ERC VAREG - {\it Variational approach to the regularity of the free boundaries} \rm (grant agreement No. 853404).}



\begin{document}

	\begin{abstract} 		
In this paper we study the local behavior of solutions to some free boundary problems. We relate the theory of quasi-conformal maps to the regularity of the solutions to nonlinear thin-obstacle problems; we prove that the contact set is locally a finite union of intervals and we apply this result to the solutions of one-phase Bernoulli problems with geometric constraint. 
We also introduce a new conformal hodograph transform, which allows to obtain the precise expansion at branch points of both the solutions to the one-phase problem with geometric constraint and a class of symmetric solutions to the two-phase problem, as well as to construct examples of free boundaries with cusp-like singularities.
	\end{abstract}
	\maketitle

	\section{Introduction} 
	
This note is dedicated to the analysis of the branch singularities arising in two different types of free boundary problems in dimension two: non-linear thin-obstacle problems and one-phase Bernoulli problems with geometric constraint. In the last part of the paper we will present some results about branch points of the two-phase problem.

Our main motivation is the description of the structure of branch points arising in free boundary problems of Bernoulli type. One model example is the following one-phase problem with geometric constraint, which for simplicity we state for nonnegative functions $u$ defined on the unit ball $B_1$ in $\R^d$:
	\begin{align*}
\Delta u=0&\quad\text{in}\quad \Omega_u\subset B_1\cap\{x_d>0\}\\
u=0&\quad\text{on}\quad B_1\cap\{x_d=0\}\\
|\nabla u|=1&\quad\text{on}\quad \partial\Omega_u\cap\{x_d>0\}\\
|\nabla u|\ge 1&\quad\text{on}\quad \partial\Omega_u\cap\{x_d=0\},
\end{align*}
in which 
$$\Omega_u:=\{u>0\}$$
and the geometric constraint is the inclusion $\Omega_u\subset B_1\cap\{x_d>0\}$.  
The (optimal) $C^{1,\sfrac12}$ regularity of the free boundary $\partial\Omega\cap B_1$ for this specific problem was proved by Chang-Lara and Savin in \cite{chang-lara-savin}. On the other hand, as in the case of other Bernoulli free boundary problems as the two-phase problem \cite{DSV} and the vectorial problem \cite{SV}, the $C^{1,\alpha}$ regularity of the free boundary $\partial\Omega_u\cap B_1$ by itself does not give any information on the contact set 
$$\partial\Omega_u\cap\{x_d=0\}\cap B_1,$$
nor the structure of its boundary, which is the set of points at which $\partial\Omega_u$ branches away from $\{x_d=0\}$. In dimension two, it is natural to expect that this set is discrete and that around each branch point the set $\{u=0\}\cap\{x_d>0\}$ forms a cusp. This is precisely the content of one of our main results, \cref{t:main}. 

\smallskip
We will study these singularities in two different ways. First we will prove that branch singularities for minimizers of a general non-linear thin-obstacle problem are isolated, using the theory of quasiconformal maps, and then we will deduce the same result for solutions of the problem above via an hodograph transform. Secondly, we will introduce a \emph{conformal hodograph transform} and use it to deduce the result directly. This second method has two advantages: it allows us to give a precise description of the cuspidal behavior of the free boundary at branch singularities and moreover, being reversible, it allows to show that solutions of the $2$-dimensional one phase problem with obstacle are in a $1$ to $1$ correspondence with solution to the thin-obstacle problem, thus producing many examples of cuspidal singularities. Finally we will describe a special situation in which our techniques apply to the study of branch points of solutions to the two-phase problem, give a precise description of isolated branch points, and explain what is the major difficulty there, which we will treat in forthcoming work.

\smallskip
We wish to remark that such precise results at branch points, that is singular points at which the tangent to the free boundary is a plane, usually with multiplicity, are quite rare. To our knowledge, the only such examples are the results of Chang on $2$-dimensional area minimizing currents (\cite{chang, DSS1, DSS2, DSS3}), of Sakai on the $2$-dimensional obstacle problem (\cite{sakai1,sakai2}), and of Lewy on the $2$-dimensional thin-obstacle problem (\cite{L}, and also \cite{gape} for a less precise result); like in the present paper, all these results are $2$-dimensional. \smallskip
 
 Our approach is similar in spirit to the results of Sakai and Lewy, and makes use of (quasi)-conformal techniques to prove both the local finiteness of the branch set and to give a precise description of the cuspidal behavior at such points. A possible alternative approach, which could also be applicable in higher dimensions, would be to look for a monotone quantity, such as the Almgren's frequency function as done for instance in the Chang's paper \cite{chang}; in fact, for some thin-obstacle problems, as for instance the one involving the classical Laplace operartor, the monotonicity of the Almgren's frequency function is known (see \cite{ACS,gape}) and can still be used to get information on the dimension of the branch set (see \cite{FoSp}). However, the operators we study are not regular enough to guarantee the monotonicity of the frequency function,
 and so we were naturally led to consider (quasi)-conformal techniques. Furthermore, our techniques have the additional benefit of yielding a very precise local description of the free-boundary at branch points (see Items (b) of \cref{t:main,thm:ass2,thm:ass4}) in a straightforward way, much simpler than the induction procedure that would be needed using the frequency function as in \cite{chang}.

%


\medskip

 

\subsection{Non-linear thin-obstacle problem} Let $B_1$ be the unit ball in $\R^2$ and let 
	$$B_1^+:=\{(x,y)\in B_1\ :\ y>0\}\qquad\text{and}\qquad B_1'=\{(x,y)\in B_1\ :\ y=0\}.$$
	We consider solutions $U\in C^1(B_1^+\cup B_1')$ of the following nonlinear thin-obstacle problem
	\begin{align}
	{\rm div}(\nabla \mathcal F(\nabla U)) =0&\quad\text{in}\quad  B_1^+\,,\label{e:equation_u}\\
	U\ge 0&\quad\text{on}\quad B_1'\,,\label{e:free-boundary-condition0}\\
	\mathcal F_2(\nabla U)=0&\quad\text{on}\quad \{U>0\}\cap B_1'\,,\label{e:free-boundary-condition1}\\
	\mathcal F_2(\nabla U)\le 0&\quad\text{on}\quad \{U=0\}\cap B_1'\,,\label{e:free-boundary-condition2}
	\end{align}
	where 
	$\mathcal F:\R^2\to\R$
	is a $C^2$-regular function.
	Our first main result is the following.

	\begin{theorem}[Non-linear thin-obstacle]\label{t:main}
	Suppose that $U\in C^1(B_1^+\cup B_1')$ is a solution to \eqref{e:equation_u}-\eqref{e:free-boundary-condition0}-\eqref{e:free-boundary-condition1}-\eqref{e:free-boundary-condition2} and that $\mathcal F:\R^2\to\R$ is $C^2$-regular function satisfying
	\begin{equation}\label{e:derivatives-of-F}
	\nabla \mathcal F(0)=0\qquad\text{and}\qquad \nabla^2\mathcal F(0)=\text{\rm Id}.
	\end{equation}
	Then, the following holds:\smallskip
	\begin{enumerate}[\rm(a)]
	\item\label{item:a} The set of branch points 
	\begin{equation}\label{e:def-S-thin}
	\mathcal S(U):=\big\{z\in B_1'\ :\ U(z)=0,\ \nabla U(z)=0\big\},
	\end{equation}
	is a discrete (locally finite) subset of $B_1'$.\smallskip
	\item\label{item:b} For every point $z_0\in\mathcal S(U)$ (without loss of generality $z_0=0$), there are:\smallskip
		\begin{itemize}
		\item a radius $r>0$ and a quasi-conformal homeomorphism $\Psi:B_r\to \Omega$,\\ between $B_r$ and an open set $\Omega\subset B_1$, such that:
		\begin{gather}
		\Psi\in W^{1,2}_{loc}(B_r;\R^2)\,,\\
		{\rm Im}(\Psi(z))\equiv 0 \quad \text{on}\quad {\rm Im}(z)\equiv 0\,, \label{e:0imag}\\
		|\Psi(z)- z|= o(|z|)\,;\label{e:closetoid}
		\end{gather}
		 \smallskip
		\item a holomorphic function $\Phi\colon B_1\to \C$ of the form 
		\begin{equation}\label{e:main-teo-expasnion-Phi}
		\Phi(z)=a z^k+O(z^{k+1})\quad\text{where}\quad k\ge 3\quad \text{and}\quad a\in\mathbb C\,;\smallskip
		\end{equation} 
		\end{itemize} 
    such that we can write the solution $U$ as
	\begin{equation}\label{e:main-teo-conclusion}
	U(z)=\text{\rm Re}\Big(\Phi\big(\Psi(z)\big)^{\sfrac12}\Big)\quad\text{for every}\quad z\in B_r(z_0)\ .
	\end{equation}
	\end{enumerate}
	\end{theorem}	
\begin{oss}[Optimal regularity]
We notice that one particular consequence of the previous theorem, is the optimal regularity for solutions of the non-linear thin-obstacle problem \eqref{e:equation_u}-\eqref{e:free-boundary-condition0}-\eqref{e:free-boundary-condition1}-\eqref{e:free-boundary-condition2}. In fact, if $U\in C^{1}(B_1^+\cup B_1')$ is as in \cref{t:main}, then from \eqref{e:main-teo-conclusion}, \eqref{e:main-teo-expasnion-Phi} and \eqref{e:closetoid} it follows that $U\in C^{1,\sfrac12}(B_1^+\cup B_1')$.
\end{oss}	

In the case of the classical thin-obstacle problem in which the operator is the Laplacian, that is $\mathcal F(x,y)=x^2+y^2$, the results \ref{item:a} and \ref{item:b} of \cref{t:main} were obtained by Lewy in \cite{L}; moreover, in this case, the claim \ref{item:a} can also be obtained by means of the Almgren's monotonicity formula (see \cite{ACS} and \cite{gape}); we also notice that for the classical thin-obstacle problem, the map $\Psi$ from \cref{t:main} is the identity. \smallskip

However, in order to apply this result to the one-phase problem described in the next subsection, we will be interested in solutions $u$ of the thin-obstacle problem with
$$
\mathcal F(x,y):=\frac{x^2+y^2}{1+y}
$$
and for which $\nabla u\in C^{0,\sfrac12}$ and no better. In particular, it is easy to check that \(U\) is a solution of an equation of the form
\[
{\rm div}(A(x)\nabla U) =0
\]
where \(A(x)\) is no better than \(C^{0,\sfrac12}\). For these type of equations the results in \cite{GarLin} can not be applied (and actually are known to fail) so in order to obtain our result we need to exploit the ``quasi-linear'' structure of the problem and our approach, based on the use of quasi-conformal maps, seems to be more suitable, although limited to dimension \(2\).

\subsection{One-phase problem with geometric constraint} Next, we consider the following one-phase problem constrained above an hyperplane, that is let $u: B_1\cap\{x_d\ge 0\}\to\R$ be a continuous non-negative function solution of the problem
	\begin{align}
	\Delta u=0\quad\text{in}\quad \Omega_u:=\{u>0\}\subset B_1,\label{e:equation_one-phase}\\
	u=0\quad\text{on}\quad B_1\cap\{x_d=0\},\label{e:u=0}\\
	|\nabla u|=1\quad\text{on}\quad \partial\Omega_u\cap\{x_d>0\},\label{e:free-boundary-condition}\\
	|\nabla u|\ge 1\quad\text{on}\quad \partial\Omega_u\cap\{x_d=0\}.\label{e:contact set condition}
	\end{align}
	In the recent paper by Chang-Lara and Savin \cite{chang-lara-savin} it was shown that if $u$ is a viscosity solution of this problem (that is, if the boundary conditions \eqref{e:free-boundary-condition} and \eqref{e:contact set condition} are intended in viscosity sense), then in a neighborhood of any contact point $x=(x',0)\in\partial\Omega_u\cap\{x_d=0\}$ the boundary $\partial\Omega_u$ is a $C^{1,\alpha}$-regular graph over the hyperplane $\{x_d=0\}$. More precisely in a neighborhood of a point $z_0\in\partial\Omega_u\cap\{x_d=0\}$, the boundary $\partial\Omega$ is a $C^{1,\sfrac{1}{2}}$-regular surfaces, that is, there are a radius $\rho>0$ and a $C^{1,\sfrac{1}{2}}$-regular function
	$$f:B_\rho'(z_0)\to[0,+\infty),$$
	such that, up to a rotation and translation of the coordinate system, we have 
	\begin{equation}\label{eq:fbgraph}
	\begin{cases}
	\begin{array}{rcl}
	u(x)>0&\quad\text{for}\quad&x\in (x',x_d)\in B_\rho(z_0) \quad\text{such that}\quad x_d>f(x');\\
	u(x)=0&\quad\text{for}\quad&x\in (x',x_d)\in B_\rho(z_0) \quad\text{such that}\quad  x_d\le f(x').
	\end{array}
	\end{cases}
	\end{equation}	
	We denote by $\mathcal C_1(u)$ the contact set of the free boundary $\partial\Omega_u$ with the hyperplane $\{x_d=0\}$
	$$
	\mathcal C_1(u):=\{x_d=0\}\cap\partial\Omega_u\,,$$
	 and by $\mathcal B_1(u)$ the set of points at which the free boundary separates from $\{x_d=0\}$ : 
	 $$\mathcal B_1(u):=\Big\{x\in \mathcal C_1(u)\ :\ B_r(x)\cap  \big(\partial\Omega_u\setminus\{x_d=0\}\big)\neq\emptyset\ \text{ for every }\ r>0\Big\}.$$
	By $\mathcal S_{1}(u)$ we denote the set of points in $\mathcal C_1(u)$ at which $u$ has gradient precisely equal to $1$
	\begin{equation}\label{e:def-S-one-phase}
	\mathcal S_{1}(u):=\big\{z\in \mathcal C_1(u)\,:\,|\nabla u|(z)=1\big\}.
	\end{equation}
	We notice that a priori the set $\mathcal C_1(u)$ is no more than a closed subset of $\{x_d=0\}$. Moreover, if at a point $x=(x',0)$ we have that $|\nabla u|(x',0)>1$, then this point is necessarily in the interior of $\mathcal C_1(u)$ in the hyperplane $\{x_d=0\}$. Thus, 
	\begin{center}
	$\mathcal S_1(u)$ contains all branch points,  $\mathcal B_1(u)\subset \mathcal S_1(u)$.
\end{center}

	\begin{theorem}[Analyticity at the branch points in the one phase problem with obstacle]\label{thm:ass2} Let $u$ be a solution of \eqref{e:equation_one-phase}--\eqref{e:contact set condition} in dimension $d=2$. Then, the following holds:
	\begin{enumerate}[\rm(a)]
	\item\label{item:aa} $\mathcal S_1(u)$ is locally finite and $\mathcal C_1(u)$ is a locally finite union of disjoint closed intervals of the axis $\{x_2=0\}$;\smallskip
	\item\label{item:bb} For every point $z_0\in\mathcal S_1(u)$, one of the following holds:\smallskip
	\begin{enumerate}[\rm(b.1)]
		\item\label{item:bb1} $z_0$ is an isolated point of $\mathcal C_1(u)$ and, in a neighborhood of $z_0$, the free boundary $\partial\Omega_u$ is the graph of an analytic function that vanishes only at $z_0$;\smallskip
		\item\label{item:bb2} $z_0$ lies in the interior of $\mathcal C_1(u)$ and there is $r>0$ such that $u$ is harmonic in $B_r(z_0)$ and $|\nabla u|>1$ at all points of $\{x_2=0\}\cap B_r(z_0)$ except $z_0$;\smallskip
		\item\label{item:bb3} $z_0$ is an endpoint of a non-trivial interval in the contact set $\mathcal C_1(u)$; moreover, there is an interval $\mathcal I_\rho=(-\rho,\rho)$ and analytic function $\phi:\mathcal I_\rho\to\R$ such that $\phi(0)>0$ and, up to setting $z_0=0$ and rotating the coordinate axis, 
			\begin{equation}\label{eq:exp}
		f(x)=
		\begin{cases}
		0 & \text{ if } x\ge  0\\
		x^{\sfrac k2}\,\phi(x) & \text{ if }x<0\,.
		\end{cases}
		\end{equation}
	\end{enumerate}
	\end{enumerate}
	\end{theorem}

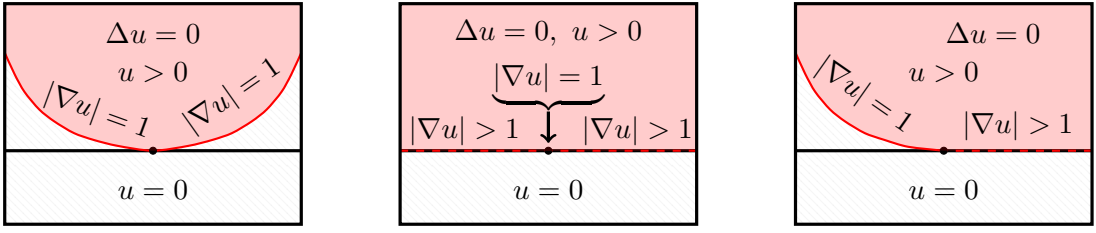
\begin{center}
\begin{tikzpicture}[rotate=0, scale= 1.3]
\begin{scope}[shift={(8,0)}]
\coordinate (O) at (0,0);
\draw[very thick] (-1.5,0) -- (1.5,0);
\draw [thick, dashed, color=red, name path=n] plot [smooth] coordinates {(1.5,0) (0,0)};
\draw [thick, color=red, name path=b] plot [smooth] coordinates {(0,0) (-0.36,0.04) (-0.6,0.11) (-0.8,0.18) (-1,0.3) (-1.15,0.42) (-1.3,0.6) (-1.4,0.78) (-1.5,1)};
\filldraw (0,0) circle (1pt);
\path[name path=2, intersection segments={of=n and b}];

\begin{scope}[transparency group,opacity=0.2]
\draw[draw=none, name path=u] (-1.5,1.5) -- (1.5,1.5);
\draw[draw=none, name path=d] (-1.5,-0.75) -- (1.5,-0.75);
\draw[draw=none, name path=0] (-1.5,-0) -- (1.5,-0);
\tikzfillbetween[of=u and 2] {color=red};
\tikzfillbetween[of=0 and d] {pattern=north west lines};
\tikzfillbetween[of=0 and 2] {pattern=north west lines};
\end{scope}

\draw[very thick] (-1.5cm,-0.75cm) rectangle ++(3cm,2.25cm);
\draw node at (0.5,1.2) {$\Delta u=0$};
\draw node at (0,0.8) {$u>0$};
\draw node at (0.75,0.25) {$|\nabla u|> 1$};
\node[label={[rotate=-35]:$|\nabla u|= 1$}] at (-0.95,0.25) {};
\draw node at (0,-0.4) {$u=0$};
\end{scope}

\begin{scope}[shift={(0,0)}]
\coordinate (O) at (0,0);
\draw[very thick] (-1.5,0) -- (1.5,0);
\draw [thick, color=red, name path=2] plot [smooth] coordinates {(-1.5,1) (-1.4,0.78) (-1.3,0.6) (-1.15,0.42) (-1,0.3) (-0.8,0.18) (-0.6,0.11) (-0.36,0.05)  (0,0) (0.36,0.05) (0.6,0.11) (0.8,0.18) (1,0.3) (1.15,0.42) (1.3,0.6) (1.4,0.78) (1.5,1)};
\filldraw (0,0) circle (1pt);

\begin{scope}[transparency group,opacity=0.2]
\draw[draw=none, name path=u] (-1.5,1.5) -- (1.5,1.5);
\draw[draw=none, name path=d] (-1.5,-0.75) -- (1.5,-0.75);
\draw[draw=none, name path=0] (-1.5,-0) -- (1.5,-0);
\tikzfillbetween[of=u and 2] {color=red};
\tikzfillbetween[of=0 and d] {pattern=north west lines};
\tikzfillbetween[of=0 and 2] {pattern=north west lines};
\end{scope}

\draw[very thick] (-1.5cm,-0.75cm) rectangle ++(3cm,2.25cm);
\draw node at (0,1.2) {$\Delta u=0$};
\draw node at (0,0.8) {$u>0$};
\node[label={[rotate=-20]:$|\nabla u|= 1$}] at (-0.65,0.05) {};
\node[label={[rotate=35]:$|\nabla u|= 1$}] at (0.95,0.25) {};
\draw node at (0,-0.4) {$u=0$};
\end{scope}

\begin{scope}[shift={(4,0)}]
\coordinate (O) at (0,0);
\draw[very thick] (-1.5,0) -- (1.5,0);
\draw [thick, dashed, color=red, name path=2] plot [smooth] coordinates {(-1.5,0) (1.5,0)};
\filldraw (0,0) circle (1pt);

\begin{scope}[transparency group,opacity=0.2]
\draw[draw=none, name path=u] (-1.5,1.5) -- (1.5,1.5);
\draw[draw=none, name path=d] (-1.5,-0.75) -- (1.5,-0.75);
\draw[draw=none, name path=0] (-1.5,-0) -- (1.5,-0);
\tikzfillbetween[of=u and 2] {color=red};
\tikzfillbetween[of=0 and d] {pattern=north west lines};
\end{scope}

\draw[very thick] (-1.5cm,-0.75cm) rectangle ++(3cm,2.25cm);
\draw node at (0,1.2) {$\Delta u=0,\ u>0$};
\draw node at (-0.85,0.2) {$|\nabla u|> 1$};
\draw node at (0.9,0.2) {$|\nabla u|> 1$};
\draw node at (0,0.65) {$\underbrace{|\nabla u|= 1}$};
\draw node at (0,-0.4) {$u=0$};
\draw[very thick,->] (0,0.4) -- (0,0.08);
\end{scope}
\end{tikzpicture}

\end{center}		
As we mentioned above we will give two proofs of this result. The first will be obtained combining \cref{t:main} with the standard hodograph transform. The second proof instead, more geometric in spirit, will be achieved via a conformal hodograph transform.
This proof has the advantage of being reversible, thus allowing us to construct examples of solutions and free boundaries  with any prescribed cuspidal behavior (without invoking any fixed point argument, as usual in the literature).

\begin{theorem}[Cuspidal points for one-phase problem]\label{thm:example-one-phase}
For any positive integer $n\in \N$, there exists a solution of \eqref{e:equation_one-phase}--\eqref{e:contact set condition} in dimension $d=2$ such that \eqref{eq:exp} in \cref{thm:ass2} holds with $k=4n-1$.
\end{theorem}

\subsection{Symmetric two-phase problem} Finally, we consider solutions to the two-phase free boundary problem in viscosity sense, that is we let $u: B_1\to\R$ be a continuous function and  we denote by $u_+$ and $u_-$ the functions 
	$$u_+=\max\{u,0\}\qquad\text{and}\qquad u_-:=\min\{u,0\}.$$
and by $\Omega_u^+$ and $\Omega_u^-$ the sets
	$$\Omega_u^{\pm}:=\{\pm u>0\}.$$ 
	Notice that with this notation $u_-$ is negative. Then $u$ is a viscosity solution of the problem
	\begin{align}
	\Delta u=0\quad\text{in}\quad \Omega_u^+\cup\Omega_u^-,\label{e:equation_2-phase}\\
	|\nabla u_+|=1\quad\text{on}\quad \partial\Omega_u^+\setminus \partial\Omega_u^-\cap B_1,\label{e:2free-boundary-condition}\\
		|\nabla u_-|=1\quad\text{on}\quad \partial\Omega_u^-\setminus \partial\Omega_u^+\cap B_1,\label{e:2free-boundary-condition2}\\
	|\nabla u_+|=|\nabla u_-|\ge 1\quad\text{on}\quad \partial\Omega_u^+\cap\partial\Omega_u^-\cap B_1.\label{e:2contact set condition}
	\end{align}
	In \cite{DSV}, we proved that if $u$ is a viscosity solution of this problem in any dimension $d\ge 2$, then in a neighborhood of any two-phase point 
	$$x_0\in \partial\Omega_u^+\cap\partial\Omega_u^-\cap B_1,$$ 
	both free boundaries $\partial \Omega_u^+\cap B_1$ and $\partial\Omega_u^-\cap B_1$ are $C^{1,\alpha}$ regular. Thus, by the classical elliptic regularity theory, also the functions $u_\pm$ are $C^{1,\alpha}$ regular respectively on $\overline\Omega_u^+\cap B_1$ and $\overline\Omega_u^-\cap B_1$ and the equations \eqref{e:equation_2-phase}-\eqref{e:2contact set condition} hold in the classical sense.
	
	We will denote with $\mathcal C_2(u_+,u_-)$ the two-phase free boundary, which is the contact set between the free boundaries $\partial\Omega_u^+$ and $\partial\Omega_u^-$, and with $\mathcal O_\pm$ the remaining one-phase parts:
	$$
	\mathcal C_2(u_+,u_-):=\partial \Omega_u^+\cap \partial \Omega_u^-\cap B_1
	\qquad \text{and}\qquad 
	\mathcal O_\pm :=\Big(\partial \Omega_u^\pm\cap B_1\Big)\setminus \mathcal C_2(u_+,u_-)\,.
	$$	
	We notice that the set $\mathcal C_2(u_+,u_-)$ is closed, while $\mathcal O_+$ and $\mathcal O_-$ are relatively open subsets respectively of $\partial\Omega_u^\pm\cap B_1$. We define the set of branch points $\mathcal B_2(u_+,u_-)$ as the set of points at which the two free boundaries $\partial\Omega_u^\pm$ separate, that is 
	\begin{equation}\label{e:definition-of-two-phase-B}
	\mathcal B_2(u_+,u_-)=\big\{x\in \mathcal C_2(u_+,u_-)\ :\ B_r(x)\cap \mathcal O_\pm\neq\emptyset\ \text{ for every }\ r>0\big\}.
	\end{equation}
	By  $C^{1}$-regularity of $u_\pm$, if $x\in\big(\partial\Omega_u^+\cup\partial\Omega_u^-\big)\cap B _1$ is such that 
	$$|\nabla u_+|(x)>1\quad\text{or}\quad |\nabla u_-|(x)>1,$$
	then it  is necessarily a two-phase non-branch point: $x\in \mathcal C_2(u_+,u_-)\setminus \mathcal B_2(u_+,u_-)$. \\
	In particular, this implies that the set
	\begin{equation}\label{e:definition-of-two-phase-S}
	\mathcal S_2(u_+,u_-):=\big\{x\in \mathcal C_2(u_+,u_-)\,:\, |\nabla u_{+}|(x)=|\nabla u_{-}|(x)=1 \big\}\,,
	\end{equation}
 contains the set of branch points $\mathcal B_2(u_+,u_-)$. \medskip

In dimension $d=2$, $\partial\Omega_u^\pm$ are locally parametrized by two $C^{1,\alpha}$ curves. 
Precisely, suppose that $z_0=(x_0,y_0)\in \mathcal C_2(u_+,u_-)$, without loss of generality we may assume that $z_0=(0,0)$, and that there is an interval  $\mathcal I_\rho:=(-\rho,\rho)$ and two $C^{1,\alpha}$-regular functions
	$$f_\pm:\mathcal I_\rho\to\R,$$
	such that
	$$f_+\ge f_-\quad\text{on}\quad \mathcal I_\rho\qquad\text{and}\qquad f_+(0)=f_-(0)=\partial_xf_+(0)=\partial_xf_-(0)=0\,,$$
	and, up to rotations and translations,
	\begin{equation}\label{eq:2phgraph}
	\begin{cases}
	\begin{array}{rll}
	u(x,y)>0&\quad\text{for}\quad  (x,y)\in \mathcal I_\rho\times \mathcal I_\rho \quad\text{such that}\quad y>f_+(x);\\
	u(x,y)=0&\quad\text{for}\quad  (x,y)\in \mathcal I_\rho\times \mathcal I_\rho \quad\text{such that}\quad  f_-(x)\le y\le f_+(x);\\
		u(x,y)<0&\quad\text{for}\quad  (x,y)\in \mathcal I_\rho\times \mathcal I_\rho \quad\text{such that}\quad y<f_-(x).
	\end{array}
	\end{cases}
	\end{equation}
	Thus, in the square $\mathcal I_\rho\times\mathcal I_\rho$, the one-phase parts $\mathcal O_+$ and $\mathcal O_-$ of the free boundary are the union of $C^{1,\alpha}$ (actually analytic) graphs over a countable family of disjoint open intervals:
	$$
	\mathcal O_\pm:=\bigcup_{i\in\N} \Gamma_\pm^i\,,
	$$
	where, for every $i\in\N$, there is an open interval $\mathcal I_i\subset\mathcal I_\rho$ such that  
	\begin{equation}\label{e:onephase}
	\Gamma_\pm^i=\big\{(x, f_\pm(x))\,:\, x\in \mathcal I_i\big\}.\quad
	\end{equation}

\begin{definition}[Symmetric solutions of the two-phase problem]\label{d:acf} In dimension $d=2$, we will say that a continuous function $u:B_1\to\R$ is a \emph{symmetric solution to the two-phase problem} if $u$ satisfies \eqref{e:equation_2-phase}-\eqref{e:2contact set condition} and moreover
	\begin{equation}\label{eq:symm}
	\mathcal H^1(\Gamma_+^i)=\mathcal H^1(\Gamma_-^i)\quad \text{for every}\quad  i \in \N\quad \text{such that}\quad \overline{\mathcal I_i}\subset \mathcal I_\rho\,.
	\end{equation}
\end{definition}
The main result of this section is the following.

\begin{theorem}[Cuspidal points for the symmetric solutions of the two-phase problem]\label{thm:ass4}
Let $u:B_1\to\R$ be a viscosity solution of the two-phase problem \eqref{e:equation_2-phase}-\eqref{e:2contact set condition}.\\ Then the following holds.
	\begin{enumerate}[\rm(a)]
	\item\label{item:aaa} If $u$ is symmetric in the sense of \cref{d:acf}, then the singular set $\mathcal S_2(u_+,u_-)$ defined in \eqref{e:definition-of-two-phase-S} is locally finite, so in particular the two-phase free boundary $\mathcal C_2(u_+,u_-)=\big(\partial\Omega_u^+\cup\partial\Omega_u^-\big)\cap B_1$ is a locally finite union of disjoint $C^{1,\alpha}$-arcs;\smallskip
	\item\label{item:bbb} If $z_0\in\mathcal S_2(u_+,u_-)$ is an isolated point of $\mathcal S_2(u_+,u_-)$, then we have one of the following possibilities:\smallskip
	\begin{enumerate}[\rm(b.1)]
		\item\label{item:bbb1} $z_0$ is an isolated point of $\mathcal C_2(u_+,u_-)$ and, in a neighborhood of $z_0$, the free boundaries $\partial\Omega_u^+$ and $\partial\Omega_u^-$ are analytic graphs meeting only in $z_0$;\smallskip
		\item\label{item:bbb2} $z_0$ lies in the interior of $\mathcal C_2(u_+,u_-)$ and moreover there is $r>0$ such that:\\ $\Delta u=0$  in $B_r(z_0)$ and $|\nabla u|>1$ at all points of $\{u=0\}\cap B_r(z_0)$ except $z_0$;\smallskip
	\item\label{item:bbb3} $z_0$ is an endpoint of a non-trivial arc in $\mathcal C_2(u_+,u_-)$, and there are an interval $\mathcal I_\rho=(-\rho,\rho)$ a constant $k\in\N$, $k\ge 3$, and an analytic function 
	$\phi:\mathcal I_\rho\to\R$ such that $\phi(0)\neq0$ and, up to setting $z_0=0$ and changing the coordinates, 
\begin{equation}\label{e:main-teo-expansion-f}	
f_{+}(x)-f_{-}(x)=
	\begin{cases}
	x^{k/2}\,\phi(|x|^{\sfrac12})&\text{if }x\le 0\\
	0&\text{if }x\ge 0\,.
	\end{cases}
	\end{equation}
	Precisely, there are analytic functions $\Phi$, $\beta_\pm$ and $\Theta$ such that for every $x\le 0$
	\begin{equation}\label{e:main-teo-precise-expansion-f}
	f_\pm(x)=\Phi\Big(x+|x|^{\sfrac52}\beta_{\pm}\big(|x|^{\sfrac12}\big)\Big)\pm\Psi\Big(x+|x|^{\sfrac52}\beta_{\pm}\big(|x|^{\sfrac12}\big)\Big),
	\end{equation}
	where $\Psi$ is of the form $\Psi(x)=|x|^{\sfrac32}\Theta(x)$.
	\end{enumerate}
	\end{enumerate}
\end{theorem} 

\begin{center}
	
	\begin{tikzpicture}[rotate=0, scale= 1.3]
	\begin{scope}[shift={(8,0)}]
	\coordinate (O) at (0,0);
	\draw [thick, dashed, color=red, name path=n] plot [smooth] coordinates {(-1.5,0.1) (-0.9,-0.06) (-0.3,0.02) (0,0)};
	\draw [thick, dashed, color=blue, name path=n2] plot [smooth] coordinates {(-0.065,0) (-0.3,0.02) (-0.9,-0.06) (-1.485,0.1)};
	\draw [thick, color=blue, name path=r] plot [smooth] coordinates {(0,0) (0.2,-0.04) (0.42,-0.105) (0.6,-0.18) (0.8,-0.29) (1.1,-0.52) (1.3,-0.77) (1.4,-0.96) (1.5,-1.2)};
	\draw [thick, color=red, name path=b] plot [smooth] coordinates {(0,0) (0.14,0) (0.36,0.03) (0.6,0.1) (0.8,0.18) (1,0.3) (1.15,0.42) (1.3,0.6) (1.4,0.78) (1.5,1)};
	\filldraw (0,0) circle (1pt);
	\path[name path=1, intersection segments={of=n and r}];
	\path[name path=2, intersection segments={of=n and b}];
	
	\begin{scope}[transparency group,opacity=0.2]
	\draw[draw=none, name path=u] (-1.5,1.25) -- (1.5,1.25);
	\draw[draw=none, name path=d] (-1.5,-1.25) -- (1.5,-1.25);
	\tikzfillbetween[of=u and 2] {color=red};
	\tikzfillbetween[of=d and 1] {color=blue};
	\end{scope}
	\draw[very thick] (-1.5cm,-1.25cm) rectangle ++(3cm,2.5cm);
	\draw node at (1,-0.01) {$u=0$};
	\draw node at (-0.7,0.17) {$|\nabla u|>1$};
	\node[label={[rotate=35]:$|\nabla u|= 1$}] at (0.95,0.25) {};
	\node[label={[rotate=-35]:$|\nabla u|= 1$}] at (0.55,-0.95) {};
	\draw node at (-0.8,0.9) {$u>0$};
	\draw node at (-0.8,-0.9) {$u<0$};
	\end{scope}
	
	\begin{scope}[shift={(0,0)}]
	\coordinate (O) at (0,0);
	\draw [thick, color=blue, name path=n] plot [smooth] coordinates {(-1.5,-1) (-1.4,-0.78) (-1.3,-0.6) (-1.15,-0.42) (-1,-0.3) (-0.8,-0.18) (-0.6,-0.1) (-0.36,-0.03) (-0.14,0) (0,0)};
	\draw [thick, color=red, name path=n2] plot [smooth] coordinates {(-1.5,1.2) (-1.4,0.96) (-1.3,0.77) (-1.1,0.52) (-0.8,0.29) (-0.6,0.18) (-0.42,0.105) (-0.2,0.04) (0,0)};
	\draw [thick, color=blue, name path=r] plot [smooth] coordinates {(0,0) (0.2,-0.04) (0.42,-0.105) (0.6,-0.18) (0.8,-0.29) (1.1,-0.52) (1.3,-0.77) (1.4,-0.96) (1.5,-1.2)};
	\draw [thick, color=red, name path=b] plot [smooth] coordinates {(0,0) (0.14,0) (0.36,0.03) (0.6,0.1) (0.8,0.18) (1,0.3) (1.15,0.42) (1.3,0.6) (1.4,0.78) (1.5,1)};
	\filldraw (0,0) circle (1pt);
	\path[name path=1, intersection segments={of=n and r}];
	\path[name path=2, intersection segments={of=n2 and b}];
	
	\begin{scope}[transparency group,opacity=0.2]
	\draw[draw=none, name path=u] (-1.5,1.25) -- (1.5,1.25);
	\draw[draw=none, name path=d] (-1.5,-1.25) -- (1.5,-1.25);
	\tikzfillbetween[of=u and 2] {color=red};
	\tikzfillbetween[of=d and 1] {color=blue};
	\end{scope}
	\draw[very thick] (-1.5cm,-1.25cm) rectangle ++(3cm,2.5cm);
	\draw node at (1,-0.01) {$u=0$};
	\draw node at (-1,0.05) {$u=0$};
	\node[label={[rotate=35]:$|\nabla u|= 1$}] at (0.95,0.25) {};
	\node[label={[rotate=35]:$|\nabla u|= 1$}] at (-0.55,-0.95) {};
	\draw node at (-0.6,0.9) {$u>0$};
	\draw node at (0.6,-0.9) {$u<0$};
	\end{scope}
	
	\begin{scope}[shift={(4,0)}]
	\coordinate (O) at (0,0);
	\draw [thick, dashed, color=red, name path=1] plot [smooth] coordinates {(-1.5,0.1) (-0.9,-0.06) (-0.3,0.1) (0,0) (0.4,-0.13) (1,0.2) (1.5,0.1)};
	\draw [thick, dashed, color=blue, name path=2] plot [smooth] coordinates {(-1.42,0.075) (-0.9,-0.06) (-0.3,0.1) (0,0) (0.4,-0.13) (1,0.2) (1.5,0.1)};
	\filldraw (0,0) circle (1pt);
	
	\begin{scope}[transparency group,opacity=0.2]
	\draw[draw=none, name path=u] (-1.5,1.25) -- (1.5,1.25);
	\draw[draw=none, name path=d] (-1.5,-1.25) -- (1.5,-1.25);
	\tikzfillbetween[of=u and 1] {color=red};
	\tikzfillbetween[of=d and 1] {color=blue};
	\end{scope}
	\draw[very thick] (-1.5cm,-1.25cm) rectangle ++(3cm,2.5cm);
	\draw node at (-0.75,0.25) {$|\nabla u|>1$};
	\node[label={[rotate=0]:$|\nabla u|> 1$}] at (0.85,-0.75) {};
	\node[label={[rotate=0]:$\underbrace{|\nabla u|=1}$}] at  (0.4,0.3) {};
	\draw[very thick,->] (0.35,0.49) -- (0.04,0.06);
	\draw node at (-0.9,0.95) {$u>0$};
	\draw node at (-0.8,-0.9) {$u<0$};
	\end{scope}
	\end{tikzpicture}
	
\end{center}


Notice that \ref{item:aaa} of the previous theorem requires that the function $u$ is symmetric in the generalized sense of \cref{d:acf}, while \ref{item:bbb3} is always true at isolated branch points. In fact, we also have the following result, which simply follows from the fact that if $z_0$ is an isolated point of $\mathcal B_2(u_+,u_-)$, then it is also an isolated point of $\mathcal S_2(u_+,u_-)$ for which \cref{thm:ass4} \ref{item:bb2} does not hold.

\begin{corollary}[Isolated cuspidal points of two-phase problem]
Let $u$ be a solution of the two-phase problem as in \cref{d:acf}. If $z_0\in \mathcal B_2(u_+,u_-)$ is an isolated point of the set $\mathcal B_2(u_+,u_-)$ defined in \eqref{e:definition-of-two-phase-B}, then at least one of the points \ref{item:bbb1} and \ref{item:bbb3} is true at $z_0$. 
\end{corollary}

\noindent We will prove \cref{thm:ass4} in \cref{ss:2phase}, where we will also discuss the obstructions in applying the conformal hodograph transform to the study of the branch points of the two-phase problem in the absence of symmetries or in the presence of weights $\lambda_\pm$ on the volume of the positivity and the negativity sets. \smallskip

Finally, as in \cref{thm:example-one-phase}, by reversing the argument from the proof of \cref{thm:ass4}, we can construct two-phase cusps with prescribed behavior. 
\begin{theorem}[Cuspidal points for two-phase problem]\label{thm:example-two-phase}
	For any positive integer $n\in \N$, there exists a solution of \eqref{e:equation_2-phase}--\eqref{e:2contact set condition} in dimension $d=2$ such that \eqref{e:main-teo-expansion-f} holds with $k=4n-1$ and \eqref{e:main-teo-precise-expansion-f} with $\Phi(x)=x^m+o(x)$, with $m\ge 2$.  
\end{theorem} 
The particular case $\Phi\equiv0$ is an immediate consequence from \cref{thm:example-one-phase} as a solution of the one-phase problem, together with its reflection, gives a solution of the two-phase one. However, the same method provides also non-symmetric examples in which the assymetry is given by the function $\Phi$.

\section{Non-linear thin-obstacle problem}
	
In this section we prove \cref{t:main} using the theory of quasi-conformal map. 

\subsection{Notation and known results}
Let $U\in C^1(B_1^+\cup B_1')$ be a solution of the thin-obstacle problem \eqref{e:equation_u}-\eqref{e:free-boundary-condition0}-\eqref{e:free-boundary-condition1}-\eqref{e:free-boundary-condition2}, where the function 
$\mathcal F:\R^2\to\R$
is $C^2$ regular.\\ We will denote by $\mathcal F_j$, $j=1,2$, and $\mathcal F_{ij}$, $1\le i, j\le 2$, the partial derivatives of $\mathcal F$.
Moreover, we identify $\R^2$ with the field of complex numbers $\mathbb C$, so we will often think of the functions on $\R^2=\mathbb C$ as functions of two real variables $(x,y)\in\R^2$ and at the same time as a function of one complex variable $z=x+iy\in \mathbb C$.
	\subsubsection{Variational inequality formulation}
	The system \eqref{e:equation_u}--\eqref{e:free-boundary-condition2} can be equivalently written in the form of a variational inequality. Precisely, the following are equivalent:
	\begin{enumerate}[\rm(1)]
	\item $U\in C^1(B_1^+\cup B_1')$ and satisfies \eqref{e:equation_u}, \eqref{e:free-boundary-condition0}, \eqref{e:free-boundary-condition1} and \eqref{e:free-boundary-condition2};
	\item $U\in H^1_{loc}(B_1^+\cup B_1')$ (that is $u\in H^1(B_r^+)$ for every $r<1$) and 
	\begin{equation}\label{e:variational_inequality}
	\int_{B_1^+}\nabla \mathcal F(\nabla U)\cdot\nabla (U-v)\,dx\le 0\quad\text{for every}\quad v\in \mathcal K_U,
	\end{equation}
	where $\mathcal K_U$ is the convex set 
	$$\mathcal K_U:=\Big\{v\in H^1_{loc}(B_1^+\cup B_1')\ :\ v\ge 0\,\text{ on }\, B_1'\,,\  v=U\,\text{ in a neighborhood of }\, \partial B_1\Big\}.$$
	\end{enumerate}	
Indeed, the implication $(1)\Rightarrow(2)$ follows simply by an integration by parts, while $(2)\Rightarrow(1)$ was proved by Frehse \cite{F}.
In particular, if $U\in H^1(B_1^+)$ minimizes the integral functional 
\begin{equation}\label{eq:nonlinearthin}
 \mathcal I(v):=\int_{B_1^+}\mathcal F(\nabla v)\,dx\,,
\end{equation}
among all functions in $\mathcal K_U$, then $U$ satisfies the variational inequality \eqref{e:variational_inequality}.

\subsubsection{Higher regularity of the solutions} 
It was proved by Frehse in \cite[Lemma 2.2]{F} that if $U\in H^1(B_1^+)$ is a solution of the variational inequality \eqref{e:variational_inequality}, then $U$ is in $H^2(B_r^+)$ for every $r<1$. Moreover, in \cite[Theorem 4.1]{DFS} it was shown that the solution $U$ is actually in $C^{1,\alpha}(B_1^+\cup B_1')$ for some $\alpha>0$.

\subsection{Local finiteness of the set of branch points}
In this subsection we prove \cref{t:main} \ref{item:a}. We introduce a special function $Q$ that we prove to be quasi-regular in the half-ball, then we obtain \cref{t:main} \ref{item:a} by applying  the Sto\"ilow's factorization theorem for quasi-conformal and quasi-regular maps (see \cite[Chapter 5]{AIM}).  \smallskip

\noindent 	Given a solution $U\colon B_1\cap\{y\geq 0\} \to \R$ of \eqref{e:equation_u}-\eqref{e:free-boundary-condition0}-\eqref{e:free-boundary-condition1}-\eqref{e:free-boundary-condition2}, we consider the function
	\begin{equation}\label{e:def-of-Q}
	Q: B_1^+\cap\{y\geq 0\} \to\mathbb C\,, \qquad Q(x+iy)=\partial_x U-i\mathcal F_2(\nabla U(x,y))
	\end{equation}
	We gather the fundamental properties of this function in the next lemma.

\begin{lemma}\label{p:Q1}
	The function $Q$ defined in \eqref{e:def-of-Q} satisfies the following properties:
	\begin{enumerate}[\rm (1)]
		\item $Q^2\in W^{1,2}({B_r^+};\mathbb C)$, for every $r<1$;
		\item there is $r_0>0$ such that, for every $r<r_0$, $Q$ satisfies the Beltrami equation
		$$
		\partial_{\bar z}Q=\mu\big(\nabla U,\nabla^2U\big)\, \partial_zQ\qquad\text{in}\qquad B_r^+\,,
		$$
		and if for some $\delta\in(0,1]$
		$$\|Id-\nabla^2\mathcal F(\nabla U(z))\|_{2}\le \delta\qquad\text{for every}\qquad z=(x,y)\in B_{r}^+\,,$$
		then 
		$$|\mu(\nabla U(z),\nabla^2U(z))|\le \frac{\delta}{2-\delta}\qquad\text{for every}\qquad z=(x,y)\in B_{r}^+\,,$$
		where for any real matrix $A=(a_{ij})_{ij}$, $\|A\|_2:=\Big(\sum_{i,j}a_{ij}^2\Big)^{\sfrac12}.$
	\end{enumerate}
\end{lemma}

\begin{oss}
Functions satisfying properties (1) and (2) are called quasi-conformal maps.
\end{oss}

\begin{proof}We first prove (1). By \cite{F}, we know that $U\in H^2(B_r^+)$ and that $|\nabla U|\in L^\infty(B_r^+)$. Thus, (1) follows directly by the definition of $Q$. Let us now prove (2).
	
	 For simplicity, we set
	$$A:=\partial_xU\qquad\text{and}\qquad B:=\mathcal F_2(\nabla U).$$
	Thus, $Q=A-iB$ and 
	$$\begin{cases}\partial_{\bar z}Q=\frac12(\partial_x+i\partial_y)(A-iB)=\frac12(\partial_xA+\partial_yB)+\frac{i}{2}(\partial_yA-\partial_xB),\smallskip\\
	\partial_{z}Q=\frac12(\partial_x-i\partial_y)(A-iB)=\frac12(\partial_xA-\partial_yB)-\frac{i}{2}(\partial_yA+\partial_xB),\end{cases}$$
	which implies 
	\begin{equation}\label{e:definition-mu}
	\begin{cases}
	\begin{array}{lll}
	4\left|\partial_{\bar z}Q\right|^2&=&(\partial_xA+\partial_yB)^2+(\partial_yA-\partial_x B)^2,\smallskip\\
	4\left|\partial_zQ\right|^2&=&(\partial_xA-\partial_yB)^2+(\partial_yA+\partial_x B)^2.
	\end{array}
	\end{cases}
	\end{equation}
We first compute 
	\begin{equation}
	\begin{cases}
	\begin{array}{lll}
	\partial_xA&=&\partial_{xx}U\\
	\partial_yA&=&\partial_{xy}U\\
	\partial_xB&=&\mathcal F_{12}(\nabla U)\partial_{xx}U+\mathcal F_{22}(\nabla U)\partial_{xy}U\\
	\partial_yB&=&\mathcal F_{12}(\nabla U)\partial_{xy}U+\mathcal F_{22}(\nabla U)\partial_{yy}U,
	\end{array}
	\end{cases}
	\end{equation}
	and, using the equation for $U$, we obtain
	\begin{equation}\label{e:estimates-mu-1}
	\begin{cases}
	\begin{array}{lll}
	\partial_xA+\partial_yB&=&\big(1-\mathcal F_{11}(\nabla U)\big)\partial_{xx}U-\mathcal F_{12}(\nabla U)\partial_{xy}U\\
	\partial_yA-\partial_xB&=&-\mathcal F_{12}(\nabla U)\partial_{xx}U+\big(1-\mathcal F_{22}(\nabla U)\big)\partial_{xy}U.
	\end{array}
	\end{cases}
	\end{equation}
	For simplicity, we use the following notation 
	$$m_{ij}:=\delta_{ij}-\mathcal F_{ij}(\nabla U)\quad\text{for every}\qquad 1\le i,j\le 2,$$
	and 
	$$\mathcal M:=\text{\rm Id}-\nabla^2\mathcal F(\nabla U)=\begin{pmatrix}m_{11} & m_{12}\\ m_{21} & m_{22}\end{pmatrix}.$$
	We also set 
	$$\|\mathcal M\|_2^2:=m_{11}^2+2m_{12}^2+m_{22}^2.$$
	Then, by \eqref{e:estimates-mu-1} and the Cauchy-Schwartz inequality, we immediately obtain 
	\begin{equation}\label{e:estimates-mu-2}
	(\partial_xA+\partial_yB)^2+(\partial_yA-\partial_xB)^2\le \|\mathcal M\|_2^2|\nabla A|^2.
	\end{equation}
	In order to estimate $\left|\partial_zQ\right|^2$ in \eqref{e:definition-mu}, we write
	\begin{align*}
	(\partial_xA-\partial_yB)^2+(\partial_yA+\partial_x B)^2&=\Big(2\partial_xA-(\partial_xA+\partial_yB)\Big)^2+\Big(2\partial_yA-(\partial_yA-\partial_x B)\Big)^2\\
	&=4|\nabla A|^2-4\nabla A\cdot \mathcal M(\nabla A)+(\partial_xA+\partial_yB)^2+(\partial_yA-\partial_xB)^2\\
	&=:4|\nabla A|^2+\mathcal R,
	\end{align*}
	where by \eqref{e:estimates-mu-1} and \eqref{e:estimates-mu-2}, we have the estimate
	$$|\mathcal R|\le \Big(4\|\mathcal M\|_2+\|\mathcal M\|_2^2\Big)|\nabla A|^2.$$
	Now, if at some point $\nabla A=0$, then  $\partial_z Q=\partial_{\bar z}Q=0$. Thus, we can define $\mu$ as follows:
$$\mu=0\,,\quad\text{if}\quad \nabla A=0\ ;\qquad
	\mu=\frac{\partial_{\bar z}Q}{\partial_zQ}\,,\quad\text{if}\quad \nabla A\neq0\,.$$
	Since $A$, $\partial_{\bar z}Q$ and ${\partial_zQ}$ are all functions of $\nabla U$ and $\nabla^2U$, also $\mu$ can be written in terms of the same variables, that is: $\mu=\mu(\nabla U,\nabla^2U)$. We notice that with this definition, $\mu$ remains bounded. Indeed,
	$$|\mu|^2=\left|\frac{\partial_{\bar z}Q}{\partial_zQ}\right|^2\le \frac{\|\mathcal M\|_2^2}{4-4\|\mathcal M\|_2+\|\mathcal M\|_2^2}=\left(\frac{\|\mathcal M\|_2}{2-\|\mathcal M\|_2}\right)^2,$$
	so that for $r$ sufficiently small the conclusion follows.
\end{proof}

\begin{proof}[\bf Proof of  \cref{t:main} \ref{item:a}]
Let $Q$ be the function defined in \eqref{e:def-of-Q} and let 
$$
S(z):=\begin{cases}
Q(z)^2& \text{if Im}(z)\geq 0\\
\overline {S}(\overline z) & \text{if Im}(z)\leq 0
\end{cases}
$$
We notice that 
$$
\text{Im}(Q^2(z))=\partial_xU\cdot\mathcal F_2(\nabla U)=0\qquad \text{on}\qquad\{\text{Im}(z)=0\}\,,
$$
so that the function $S$ is in $W^{1,2}(B_{r})$ and satisfies the Beltrami equation 
$$
\partial_{\bar z}S=\psi (z)\, \partial_zS\qquad\text{in}\qquad B_r^+\,,
$$
where 
$$
\psi(z)=\psi(x+iy):=\begin{cases}
\mu\big(\nabla U(x,y),\nabla^2U(x,y)\big)&\  \text{if}\quad\text{\rm Im}(z)\geq 0\,,\\
\overline {\psi}(\overline z) & \ \text{if}\quad \text{\rm Im}(z)\leq 0\,.
\end{cases}
$$
Thus, by \cite[Theorem 5.5.2]{AIM}, we get the claim.\qedhere
\end{proof}

\subsection{Local behavior of the solutions at branch points}
In this subsection we prove \cref{t:main} \ref{item:b}. Given a branch point $z_0\in\mathcal S$, we construct a quasi-regular mapping whose real part is precisely the solution $U$. Without loss of generality, we assume that $z_0=0$ and we choose a radius $r>0$ such that 
\begin{equation}\label{e:conf}
\{U=0\}\cap B_r'=\{x\le0\}\cap B_r'\qquad\text{and}\qquad \{U>0\}\cap B_r'=\{x>0\}\cap B_r'.
\end{equation}
We now notice that the differential form 
$$\alpha=-\mathcal F_2(\nabla U)\,dx+\mathcal F_1(\nabla U)\,dy$$
is closed in $B_r^+$ and so the potential
$$V: B_r^+\cup B_r'\to \R\ ,\qquad V(x,y):=\int_0^1\Big(-\mathcal F_2\big(\nabla U(tx,ty)\big)x+\mathcal F_1\big(\nabla U(tx,ty)\big)y\Big)\,dt$$
is Lipschitz continuous in $B_r^+\cup B_r'$, $C^2$ in $B_r^+$ and satisfies 
$$\begin{cases}
\begin{array}{rcl}
\partial_xV=-\mathcal F_2(\nabla U)&\quad \text{in}\quad &B_r^+,\\
\partial_yV=\mathcal F_1(\nabla U)&\quad \text{in}\quad &B_r^+,\\
UV=0&\quad \text{on}\quad & B_r'\,,
\end{array}
\end{cases}$$
where the last equality follows from \cref{e:conf} and the very definition of \(V\). We next define the complex function
\begin{equation}\label{e:def-of-P}
P: B_r^+\cap\{y\geq 0\} \to\mathbb C\,, \qquad P(x+iy)=U(x,y)+iV(x,y).
\end{equation}
\begin{oss}
Notice that, by the definition of $V$, we have $\partial_xP=Q$ in $B_r^+$. 	
\end{oss}	

We now prove the following lemma. 

\begin{lemma}\label{p:P}
	The function $P$ defined in \eqref{e:def-of-P} satisfied the following properties.
	\begin{enumerate}[\rm (1)]
		\item $P^2\in W^{1,\infty}_{\text{loc}}(B_1^+\cup B_1')$, 
		\item $P$ satisfies the Beltrami equation
		\begin{equation}\label{e:beltrami-eta}
		\partial_{\bar z}P=\eta(\nabla U)\, \partial_zP\qquad\text{in}\qquad B_r^+\,,	
		\end{equation}
where $\eta(\nabla U)=o(|\nabla U|)$. 
\end{enumerate}
\end{lemma}

\begin{proof}
The first claim follows from the Lipschitz continuity of $U$ and $V$. In order to prove the second claim, we compute
 $$\begin{cases}2\partial_{\bar z}P=(\partial_x+i\partial_y)(U+iV)
 =(\partial_xU-\mathcal F_1(\nabla U))+{i}(\partial_yU-\mathcal F_2(\nabla U)),\smallskip\\
 2\partial_{z}P=(\partial_x-i\partial_y)(U+iV)
 =(\partial_xU+\mathcal F_1(\nabla U))-{i}(\partial_yU+\mathcal F_2(\nabla U)),\end{cases}$$
Now, by the differentiability of $\mathcal F_1$ and $\mathcal F_2$ in zero and \eqref{e:derivatives-of-F}, we can write 
$$\mathcal F_1(X)-X_1= \eps_1(X)|X|\qquad\text{and}\qquad \mathcal F_2(X)-X_2= \eps_2(X)|X|,$$
for every $X=(X_1,X_2)\in\R^2$, where the functions $\eps_1$ and $\eps_2$ are such that
$$\lim_{|X|\to0}\eps_1(X)=\lim_{|X|\to0}\eps_2(X)=0\,,$$
from which the first part of the claim follows. 

\end{proof}	

\begin{proof}[\bf Proof of  \cref{t:main} \ref{item:b}]
	Let $P$ be the function defined in \eqref{e:def-of-P} and let 
	$$
	T(z):=\begin{cases}
	P(z)^2& \text{if Im}(z)\geq 0\\
	\overline {T}(\overline z) & \text{if Im}(z)\leq 0
	\end{cases}
	$$
Then
	$$
	\text{Im}(P^2(z))=U(z)V(z)=0\qquad \text{on}\qquad\{\text{Im}(z)=0\}\,,
	$$
	so $T$ is Lipschitz continuous on $B_{r}$, and satisfies the Beltrami equation 
	\begin{equation}\label{e:beltramiT}
	\partial_{\bar z}T=\phi(z)\, \partial_zT\quad\text{in}\quad B_r\,,
	\end{equation}
	where $\phi$ is the extension over the whole $B_r$ of the Beltrami coefficient $\eta(\nabla U)$ from \eqref{e:beltrami-eta} :
	$$
	\phi(z)=\phi(x+iy):=\begin{cases}
	\eta(\nabla U(x,y))&\  \text{if}\quad\text{\rm Im}(z)\geq 0\,,\\
	\overline {\phi}(\overline z) & \ \text{if}\quad \text{\rm Im}(z)\leq 0\,.
	\end{cases}
	$$
	Using again \cite[Theorem 5.5.1 and Corollary 5.5.3]{AIM}, we conclude that there exist an homeomorphism $\Psi \in W^{1,2}(B_r;B_1)$, solution of  \eqref{e:beltramiT} and such that $\Psi(0)=0$ and $\Psi(\rho)=\rho$, for some $\rho<r$,
	and an holomorphic function $\Phi \colon \Omega \to \C $ such that 
\begin{equation}\label{e:stoilow_for_T}
	T(z)= \Phi(\Psi(z)) \qquad \forall z\in B_{r}\,.
\end{equation}
	
\noindent	Next we prove \eqref{e:0imag}. Observe that if $\Psi$ is a solution to \eqref{e:beltramiT}, then also $\overline \Psi(\overline z)$ is a solution to \eqref{e:beltramiT}, and moreover $\overline \Psi(0)=\Psi(0)=0$ and $\overline \Psi(\rho)=\Psi(\rho)=1$. It follows, by uniqueness of normalized solutions, that $\overline \Psi(z)=\Psi(z)$, which implies \eqref{e:0imag}.\smallskip
	

\noindent	Finally we come to \eqref{e:closetoid}. Suppose by contradiction that  \eqref{e:closetoid} is false. Then, there is a sequence of radii $\rho_k\to 0$ such that the sequence of homeomorphisms $\Psi_k \in W^{1,2}(B_r, B_1)$, solutions of
	$$
	\partial_{\bar z} \Psi_k=\phi(z)\, \partial_z \Psi\ \text{ in }\ B_r\,,\quad \Psi_k(0)=0\,,\quad\Psi_k(\rho_k)=\rho_k\,,
	$$
	doesn't converge uniformly to the function $z$. Consider the sequence of functions $\tilde \Psi_k(z):= \rho_k^{-1}\,\Psi_k(\rho_k\,z)$, then they are solutions of 
	$$
	\partial_{\bar z} \tilde\Psi_k=\phi\left(\rho_k\,z\right)\, \partial_z \tilde\Psi\,\text{ in }B_{r/\rho_k}\,\quad \tilde\Psi_k(0)=0\,,\quad\tilde\Psi_k(1)=1\,.
	$$
	{Reasoning as in the proof of \cref{p:P} and using the fact that $\nabla U(\rho_k z) \to 0$ as $k\to \infty$, since $U\in C^1$ and $\nabla U(0)=0$, we have
	$$
	\lim_{k\to 0}\phi\left(\rho_k\,z\right)=0\qquad \text{a.e. }z\in B_{r/\rho_k}\,.
	$$}
	Using \cite[Lemma 5.3.5]{AIM}, we have that $\tilde \Phi_k$ converges locally uniformly to a homeomorphism $\tilde\Psi:\mathbb C\to\mathbb C$, which is a solution of  
	$$
	\partial_{\bar z} \tilde\Psi=0\, \text{ in }\C,\quad \tilde\Psi(0)=0\,,\quad\tilde\Psi(1)=1\,.
	$$
	But this implies that $\tilde\Psi(z)=z$, which is a contradiction for $k$ sufficiently large.

{In particular notice that, if $  \Phi(z)=z^k+O(z^{k+1})$, then the $C^{1}$ regularity of solutions to the non-linear thin-obstacle problem (see for instance \cite{Fe}) implies that $k\geq 3$.}\qedhere
\end{proof}

\section{\cref{thm:ass2}: proof via quasiconformal maps}
	
	In this section, we will prove Theorem \ref{thm:ass2} as a consequence of Theorem \ref{t:main} combined with an application of the hodograph transform.

	\subsection{The hodograph transform}
	In this section we write the hodograph transformation of a solution $u$ of \eqref{e:first_equation}--\eqref{e:third_equation}. We do this in every dimension $d\ge 2$. 
	
	\subsubsection{Notation}We adopt the following notation. We write every point $x\in\R^d$ in coordinates as $x=(x',x_d)\in \R^{d-1}\times\R$. For every $\rho>0$, we denote by $B_\rho$ and $B_\rho'$ the balls centered in zero of radius $\rho$ in $\R^d$ and $\R^{d-1}$, respectively. We will identify $\R^{d-1}$ with the hyperplane $\R^{d-1}\times\{0\}\subset\R^d$, thus
	$$B_\rho'=B_\rho\cap\{x_d=0\}\qquad \text{and}\qquad B_\rho^+=B_\rho\cap\{x_d>0\}.$$ 
	We denote by $\nabla_{x'}$ the gradient with respect to the first $d-1$ coordinates $x'=(x_1,\dots,x_{d-1})$. Thus, for every function $u:\R^d\to\R$, we can write the full gradient $\nabla u$ as
	$$\nabla u=(\nabla_{x'}u,\partial_du)\quad\text{and}\quad |\nabla u|^2=|\nabla_{x'}u|^2+|\partial_du|^2.$$
Let us assume that $0\in \mathcal S_1(u)$, that is $0$ is a branch point, and let $f\in C^{1,\alpha}$ be the function that locally describes the free-boundary $\partial \Omega_u$ as in \eqref{eq:fbgraph}, so that
	\begin{equation*}
	f(0)=0\quad\text{and}\quad \nabla_{x'}f(0)=0.
	\end{equation*}
	Now since $u(x',f(x'))$ vanishes for evey $x'\in B_\rho'$, we have that $\nabla_{x'}u(0)=0$. Thus
	$$\nabla u(0)=\partial_d u(0)\,e_d\qquad\text{and}\qquad\partial_du_+(0)\ge 1\ .$$
	
	\subsubsection{The hodograph transform}Let $0\in \partial\Omega_u\cap \{x_d=0\}$ and $f:B_\rho'\to[0,+\infty)$ be as above. We consider the change of coordinates 
	$$y'=x'\,,\quad y_d=u(x',x_d).$$
	Since $u\in C^{1,\alpha}(\overline \Omega_u\cap B_1)$, and since $\partial_ d u(0)\ge 1>0$, we have that the function 
	$$T:B_\rho\cap\overline\Omega_u\to\R^d\cap\{x_d\ge 0\}\ ,\qquad T(x',x_d)=(y',y_d),$$
	is invertible for $\rho$ small enough. In particular, the set $\ T\big(B_\rho\cap\overline\Omega_u\big)\ $ is an open neighborhood of $0$ in the upper half-space $\R_d\cap\{x_d\ge 0\}$. Let 
	$$S:T\big(B_\rho\cap\overline\Omega_u\big)\to B_\rho\cap\overline\Omega_u\  ,\qquad S(y',y_d)=(x',x_d),$$
	be the inverse of $T$. Since the map $T$ does not change the first $d-1$ coordinates, there is a $C^{1,\alpha}$ regular function $v$, defined on the set $\,T\big(B_\rho\cap\overline\Omega_u\big)\,$, such that 
	$$S(y',y_d)=\big(y',v(y',y_d)\big).$$
	We will write this in coordinates as 
	$$x'=y'\ ,\quad x_d=v(y',y_d).$$
	
	\begin{oss}
		The function $v$ contains all the information of the free boundary $\partial\Omega_u$. Precisely, for every $x'$ in a neighborhood of $0\in\R^{d-1}$, we have 
		\begin{equation}\label{e:v+f}
		v(x',0)=f(x').
		\end{equation}
		Indeed, it is immediate to check that for any point $(x',x_d)$ in a neighborhood of zero,
		$$x_d=f(x')\ \Leftrightarrow\ (x',x_d)\in\partial\Omega_u\ \Leftrightarrow\ x_d=v(x',u(x',x_d))=v(x',0).$$
		As a consequence of \eqref{e:v+f}, we get that 
		\begin{equation}\label{e:order}
		v(x',0)\ge 0\quad\text{for every $x'$ in a neighborhood of zero in $\R^{d-1}$}.
		\end{equation}
	\end{oss}

	\begin{lemma}[Hodograph transform]\label{lem:hodo}
		Let $u$, $T$, $B_\rho$ and $v$ be as above. Then, there is $r>0$ such that  
		$$B_r\cap \{x_d\ge 0\}\subset T\big(B_\rho\cap\overline\Omega_u\big),$$
		and such that the function
		$$v:B_r\cap\{x_d\ge 0\}\to\R,$$
		exists, is $C^{1,\alpha}$ in $B_r\cap\{x_d\ge 0\}$ and $C^{\infty}$ in $B_r\cap\{x_d> 0\}$. 
		Moreover, the function
		$$w:B_r\cap\{x_d\ge 0\}\to\R\ ,\qquad w(x',x_d)=v(x',x_d)-x_d$$
		solves the nonlinear thin-obstacle problem 
			\begin{align}
		{\rm div}(\nabla \mathcal F(\nabla w)) =0&\quad\text{in}\quad  B_r^+\,,\label{e:first_equation}\\
		w\ge 0&\quad\text{on}\quad B_r'\,,\label{e:order:B_r'}\\
		\mathcal F_d(\nabla w)=0&\quad\text{on}\quad \{w>0\}\cap B_r'\,,\label{e:second_equation}\\
		\mathcal F_d(\nabla w)\le 0&\quad\text{on}\quad \{w=0\}\cap B_r'\,,\label{e:third_equation}
		\end{align}
		for the nonlinearity 
		$\displaystyle\mathcal F(x',x_d):=\frac{|x'|^2+x_d^2}{1+x_d}.$	
	\end{lemma}	
\begin{oss}\label{rem:nonlin2}
We notice that \eqref{e:v+f} implies that the contact sets of the solution of the one-phase problem $u$ and the solution of the nonlinear thin-obstacle problem $w$  are mapped one into the other:
$$\mathcal C_1(u)=\partial\Omega_u\cap B_r'= S(\{w=0\}\cap B_r')
$$
as well as the singular sets defined in \eqref{e:def-S-thin} and \eqref{e:def-S-one-phase}
$$\mathcal S_1(u)=B_r'\cap\{u=0\}\cap\{|\nabla u|=1\}=S(B_r'\cap\{w=0\}\cap\{|\nabla w|=0\}).$$
\end{oss}	
	\begin{proof}[Proof of \cref{lem:hodo}]
		We first notice that 
		$$w(x',0)=v(x',0)=f(x')\quad\text{for every}\quad x'\in B_r'.$$
		This proves \eqref{e:order:B_r'} and the first part of \eqref{e:third_equation}. Next, we notice that since 
		$$v\big(x',u(x',x_d)\big)=x_d\quad\text{for every}\quad (x',x_d)\in B_\rho\cap \overline\Omega_u,$$
		we have that
		\begin{equation}\label{e:partial_i}
		\partial_{i} v_+\big(x', u_+(x',x_d)\big)+\partial_{d}v_+\big(x', u(x',x_d)\big)\partial_{i} u_+(x',x_d)=0\quad\text{for}\quad i=1,\dots,d-1,
		\end{equation}
		and 
		\begin{equation}\label{e:partial_d}
		\partial_d v\big(x', u(x',x_d)\big)\partial_{d} u(x',x_d)\equiv 1.
		\end{equation}
		Thus, we can compute
		\begin{equation}\label{e:partial_dd}
		\Big(1+\partial_d w\big(x', 0\big)\Big)\partial_{d} u(x',f(x'))\equiv 1,
		\end{equation}
		and since $\partial_du(x',0)\ge 1$, we obtain also the second part of \eqref{e:third_equation}.
		
		Next, in order to prove that the boundary condition \eqref{e:second_equation} holds, we notice that it is equivalent to 
		$$\big(\partial_dv(x',0)\big)^2=1+|\nabla_{x'}f(x')|^2\quad\text{for}\quad x'\in B_r'\cap\{f>0\},$$
		and, in view of \eqref{e:partial_dd}, also to 
		$$\big(\partial_{d} u(x',f(x'))\big)^2\Big(1+|\nabla_{x'}f(x')|^2\Big)=1\quad\text{for}\quad x'\in B_r'\cap\{f>0\},$$
		which is a consequence of the identity 
		$$\partial_iu(x',f(x'))+\partial_du(x',f(x'))\partial_if(x')\equiv 0\quad\text{for every}\quad i=1,\dots,d-1,$$
		and the boundary condition
		$$(-\nabla_{x'} f(x'),1)\cdot\nabla u(x',f(x'))=-\big(|\nabla_{x'}f(x')|^2+1\big)^{\sfrac12}\quad\text{on}\quad \{f>0\}.$$
			In order to prove  \eqref{e:first_equation} we notice that, in $\Omega_u$, $u$ is a local minimizer of the Dirichlet integral 
			$$\displaystyle \ J(u)=\int |\nabla u|^2\,dx\,, $$ 
		which can be expressed in terms of $w$ by applying \eqref{e:partial_i} and \eqref{e:partial_d}:
		$$|\nabla u|^{2} (x',x_d)=\frac{|\nabla_{x'}  v|^2\big(x', u(x',x_d)\big)+1}{|\partial_d v|^2\big(x', u(x',x_d)\big)}\qquad\text{and}\qquad \det (\nabla T)(x',x_d)=\partial_{d}u(x',x_d).$$
		Now, by the change of coordinates $y'=x'$, $y_d=u(x',x_d)$, we get
		$$\int_{B_\rho\cap\overline\Omega_u} |\nabla u|^2\,dx
		=\int \frac{|\nabla_{y'} v|^2(y',y_d)+1}{|\partial_d v|^2(y',y_d)}\frac{1}{|\partial_d u(x',x_d)|}\,dy=\int \frac{|\nabla_{y'} v|^2(y',y_d)+1}{\partial_d v(y',y_d)}\,dy$$
		where all the integrals in $dy$ are over $T(B_\rho\cap\overline\Omega_u)$. Now, by the definition of $w$, we get
		$$\int_{B_\rho\cap\Omega_u} |\nabla u|^2\,dx
		=\int_{T(B_\rho\cap\overline\Omega_u)} \left(\frac{|\nabla w|^2(y', y_d)}{1+\partial_d w(y', y_d)}+2\right)\,dy.$$
		Thus, $w$ minimimizes the functional
		$$J(w)=\int \frac{|\nabla w|^2(y', y_d)}{1+\partial_d w(y', y_d)}\,dy$$
		in the open set $T(B_\rho\cap\Omega_u)$
		with respect to perturbations of the form $w+\eps\varphi$ for small $\eps$ and smooth $\varphi$. This concludes the proof of \eqref{e:first_equation}.
	\end{proof}	
	
\begin{proof}[Proof of \cref{thm:ass2}] \cref{thm:ass2} follows by combining  \cref{lem:hodo} with \cref{t:main}.
	
\end{proof}
%

\section{\cref{thm:ass2,thm:example-one-phase}: proof via conformal hodograph transform}

In this section we prove \cref{thm:ass2} by introducing a new, conformal version, of the hodograph transform, which not only provides another proof of the fact that the one-phase branch points are isolated, but also provides the full expansion of the solution, and a way to construct examples of solutions with prescribed vanishing order (see \cref{thm:example-one-phase}).

\subsection{The harmonic conjugate}\label{ss:har} Let $u$ be a solution of the one-phase problem \eqref{e:equation_one-phase}--\eqref{e:contact set condition}, let $\mathcal S_1(u)$ be the singular set defined in \eqref{e:def-S-one-phase} and let $0\in \mathcal S_1(u)$. Let $\mathcal I_\rho=(-\rho,\rho)$ and let $f:\mathcal I_\rho\to\R$ be the $C^{1,\alpha}$ function from \eqref{eq:fbgraph} that describes locally the free boundary $\partial \Omega_u\cap B_{\rho}$; we recall that $f$ is non-negative and
$f(0)=f'(0)=0.$
	Now, since the function 
	$$\mathcal I_\rho\ni x\mapsto u(x,f(x)),$$ 
	vanishes for every $x\in \mathcal I_\rho$, we have that $\partial_{x}u(0,0)=0$. Thus
	$$\nabla u(0,0)=\partial_y u(0,0)\,e_2\qquad\text{and}\qquad\partial_y u(0,0)\ge 1\ ,$$
where $e_2=(0,1)$. We next define the open set 
$$\Omega_\rho=\Big\{(x,y)\in\mathcal I_\rho\times \mathcal I_\rho\ :\ f(x)>y\Big\},$$
and the boundary 
$$\Gamma_\rho:=\Big\{(x,y)\in\mathcal I_\rho\times \mathcal I_\rho\ :\ f(x)=y\Big\}.$$
Since $\Omega_\rho$ is simply connected, and $u$ is harmonic in $\Omega_\rho$, there is a function 
$$U: \Omega_\rho\cup \Gamma_\rho\to\R$$
which solves the problem
$$U(0,0)=0,\quad \partial_xU=\partial_yu\quad\text{and}\quad \partial_yU=-\partial_xu\quad\text{in}\quad \Omega_\rho.$$
We recall that, for any $(x,y)\in \Omega_\rho\cup \Gamma_\rho$, $U(x,y)$ is the line integral $\,\displaystyle\int_\sigma \alpha\,$ of the $1$-form 
$$\alpha:=\partial_y u(x,y)\,dx-\partial_xu(x,y)\,dy$$
over any curve
$$\sigma:[0,1]\to\Omega_\rho\cup \Gamma_\rho$$
connecting the origin $(0,0)$ to $(x,y)$. In particular, $U$ is as regular as $u$: 
$$U\in C^{1,\alpha}(\Omega_\rho\cup \Gamma_\rho).$$
If we choose $\sigma$ to be the curve parametrizing the free boundary $\Gamma_\rho$, 
$$\sigma:[0,x]\to\R^2,\quad \sigma(t)=(t,f(t)),$$ 
then, by integrating $\alpha$ over $\sigma$ and using that 
$$\partial_x u(t,f(t))+f'(t)\partial_y u(t,f(t))=0\qquad\text{for every}\qquad t\in\mathcal I_\rho\ ,$$ 
we obtain the formula 
\begin{align*}
U(x,f(x)):&=\int_0^x\Big(\partial_yu(t,f(t))-\partial_xu(t,f(t))f'(t)\Big)\,dt\\
&=\int_0^x|\nabla u|(t,f(t))\sqrt{1+f'(t)^2}\,dt=\int_\sigma|\nabla u|.
\end{align*}
In what follows, we will use the notation 
$$\eta(x):=U(x,f(x))=\int_\sigma|\nabla u|.$$

	\subsection{The conformal hodograph transform}\label{ss:hodotrans} With the notation from \cref{ss:har}, we consider the change of coordinates 
	$$x'=U(x,y)\,,\quad y'=u(x,y)\,,$$
	given by the $C^{1,\alpha}$-regular map 
	$$T:\Omega_\rho\cup\Gamma_\rho\to\R^2\cap\{y'\ge 0\}\,,\qquad T(x,y)=(x',y')\,.$$
	Now, by the definition of $U$ and the fact that  $\partial_ y u(0,0)\ge 1$, we have that the map  $T$ is invertible for $\rho$ small enough. In particular, the set $\ T\big(\Omega_\rho\cup\Gamma_\rho\big)\ $ is an open neighborhood of $(0,0)$ in the upper half-plane $\R^2\cap\{y'\ge 0\}$. Let 
	$$S:T\big(\Omega_\rho\cup\Gamma_\rho\big)\to \Omega_\rho\cup\Gamma_\rho\  ,\qquad S(x',y')=(x,y)\,,$$
	be the inverse of $T$. We can write $S$ as
	$$S(x',y')=\big(V(x',y'),v(x',y')\big),$$
which in coordinates reads as
	$$x=V(x',y')\,,\quad y=v(x',y')\,.$$
As in the case of the classical hodograph transform, the function $v$ contains all the information of the free boundary $\Gamma_\rho$. Precisely, for every $x\in \mathcal I_\rho$, we have 
$$y=f(x)\ \Leftrightarrow\ (x,y)\in\Gamma_\rho\ \Leftrightarrow\ y=v\big(U(x,y),u(x,y)\big)=v(x',0).$$
		As a consequence, we obtain the equation 
		$$f(x)=v(\eta(x),0)\quad\text{for every}\quad x\in\mathcal I_\rho\,.$$
		In particular, for $x'\in\R$ in a neighborhood of zero, $v(x',0)\ge 0$ and 
		\begin{equation}\label{e:order}
		v(x',0)> 0\quad\Leftrightarrow\quad f(\eta^{-1}(x'))>0.
		\end{equation}
\begin{oss}\label{oss:contact-set}
We notice that, in terms of the contact sets 
$$\mathcal C_1(u)=\{y=0\}\,\cap\,\partial\Omega_u\qquad\text{and}\qquad \mathcal C(v)=\{y'=0\}\,\cap\,\{v(x',0)=0\},$$
the map $\eta$ is locally a $C^1$ diffeomorphism, which is sending $\mathcal C_1(u)$ into $\mathcal C(v)$.
\end{oss}

	\begin{lemma}[Equations for $v$]\label{lem:con1ph}
		Let $T=(U,u)$ and $S=(V,v)$ be as above.\\ Then, there is $r>0$ such that  
		$$B_r\cap \{y'\ge 0\}\subset T\big(\Omega_\rho\cup\Gamma_\rho\big),$$
		and such that the function
		$$v:B_r\cap\{y'\ge 0\}\to\R,$$
		is  $C^{1,\alpha}$-regular in $B_r\cap\{y'\ge 0\}$ and  $C^{\infty}$ in $B_r\cap\{y'> 0\}$.\\ 
		Moreover, if we denote by $\mathcal C_v$ the contact set 
		\begin{equation}
		\mathcal C_v:=\Big\{(x',0)\,:\,x'=\eta(x),\ x\in \mathcal I_\rho,\ f(x)=0\Big\},\label{e:contact-set-v}
		\end{equation}
		then  $v$ solves the problem 
\begin{align}
\Delta v=0&\quad\text{in}\quad B_r\cap\{y'>0\},\label{e:system00-w:equation-w}\\
v\ge 0&\quad\text{on}\quad B_r\cap\{y'=0\},\label{e:system00-w:sign-w}\\
|\nabla v|= 1&\quad\text{on}\quad B_r\cap\{y'=0\}\setminus\mathcal C_v,\label{e:system00-w:boundary-condition-w>0}\\
v=0\quad\text{and}\quad |\nabla v|\le 1&\quad\text{on}\quad B_r\cap\{y'=0\}\cap\mathcal C_v.\label{e:system00-w:boundary-condition-w=0}
\end{align}
Moreover, for every $x\in \Gamma_\rho$, we have the identities
\begin{equation}\label{eq:etauu}
f'(x)= \frac{\partial_{x'}v(\eta(x),0)}{\partial_{y'}v(\eta(x),0)}
\qquad \text{and}\qquad 
\eta'(x)=\frac{1}{\partial_{y'}v(\eta(x),0)}.
\end{equation}
	\end{lemma}

	\begin{proof}  We start by proving that $v$ satisfies the equations \eqref{e:system00-w:equation-w}--\eqref{e:system00-w:boundary-condition-w=0}. First notice that \(v\) is harmonic since it is the second component of a conformal map. Moreover, since
		$$v\big(U(x,y),u(x,y)\big)=y\quad\text{for every}\quad (x,y)\in \Omega_\rho,$$
		taking the derivatives with respect to $x$ and $y$, we obtain that
		\begin{align*}
		\partial_{x'} v\big(U(x,y),u(x,y)\big)\partial_xU(x,y)+	\partial_{y'} v\big(U(x,y),u(x,y)\big)\partial_xu(x,y)=0,\\
	\partial_{x'} v\big(U(x,y),u(x,y)\big)\partial_yU(x,y)+	\partial_{y'} v\big(U(x,y),u(x,y)\big)\partial_yu(x,y)=1.
	\end{align*}
	By exploiting that  $\partial_xU=\partial_yu$ and $\partial_yU=-\partial_xu$, we get
	 \begin{align}
	 \partial_{x'} v(x',y')\,\partial_yu(x,y)+	\partial_{y'} v(x',y')\,\partial_xu(x,y)=0,\label{e:partial_x2}\\
	 -\partial_{x'} v(x',y')\,\partial_xu(x,y)+	\partial_{y'} v(x',y')\,\partial_yu(x,y)=1.\label{e:partial_y2}
	 \end{align}
	 Solving the system \eqref{e:partial_x2}-\eqref{e:partial_y2} leads to
	 \begin{equation}\label{e:boh}
	 \partial_{y'}v(x',y')=\frac{\partial_yu(x,y)}{|\nabla u|^2(x,y)}
	 \qquad \text{and}\qquad
	 \partial_{x'}v(x',y')=-\frac{\partial_xu(x,y)}{|\nabla u|^2(x,y)}\,.
	 \end{equation}
Thus, we obtain
	 \begin{equation}\label{e:gradv}
	 |\nabla u|(x,y)\, |\nabla v|(x',y')=1\,,
	 \end{equation}
which gives both \eqref{e:system00-w:boundary-condition-w=0} and \eqref{e:system00-w:boundary-condition-w>0}.
We next prove \eqref{eq:etauu}. Using that $u(x,f(x))\equiv 0$, we get 
	 $$
	 f'(x)=- \frac{\partial_{x}u(x,f(x))}{\partial_{y}u(x,f(x))}\,,
	 $$
	 which together with \eqref{e:boh} gives the first part of \eqref{eq:etauu}. For the second part, we notice that the identity $v(\eta(x),0)=f(x)$ gives that
	 $$f'(x)=\eta'(x)\partial_{x'}v(\eta(x),0)\,,$$
	 which, combined with the first identity in \eqref{eq:etauu}, concludes the proof.
	 	\end{proof}

\subsection{Proof of \cref{thm:ass2}}
Let $v$ be as in the previous section and let
$$Q:=\partial_{z'} v=\partial_{x'}v-i\partial_{y'}v,$$
where $z'=x'+iy'$. Since $v$ satisfies \eqref{e:system00-w:equation-w}-\eqref{e:system00-w:boundary-condition-w=0}, we get that
\begin{equation*}
\begin{cases}
\partial_{\bar z'}Q=0\quad\text{in}\quad B_r\cap\{y'>0\},\\
|Q|= 1\quad\text{on}\quad B_r\cap\{y'=0\}\setminus\mathcal C_v,\\
\text{Re}\,Q=0\quad\text{on}\quad B_r\cap\{y'=0\}\cap\mathcal C_v,
\end{cases}
\end{equation*}
where the set $\mathcal C_v$ was defined in \eqref{e:contact-set-v}.
Consider now the function
$$P=-i\frac{Q+i}{Q-i}=-i\frac{(Q+i)(\bar Q+i)}{|Q-i|^2}=\frac{2\,\text{Re}\,Q}{|Q-i|^2}-i\frac{|Q|^2-1}{|Q-i|^2}\,.$$
Then, we have that $P(0)=0$ and 
\begin{equation*}
\begin{cases}
\partial_{\bar z'}P=0\quad\text{in}\quad B_r\cap\{y'>0\},\\
\text{Re}\, P= 0\quad\text{on}\quad B_r\cap\{y'=0\}\cap\mathcal C_v,\\
\text{Im}\, P=0\quad\text{on}\quad B_r\cap\{y'=0\}\setminus\mathcal C_v,
\end{cases}
\end{equation*}
which implies that $P^2(0)=0$ and 
\begin{equation*}
\begin{cases}
\partial_{\bar z'}(P^2)=0\quad\text{in}\quad B_r\cap\{y'>0\},\\
\text{Im}\, (P^2)= 0\quad\text{on}\quad B_r\cap\{y'=0\}.
\end{cases}
\end{equation*}
As a consequence, the zero set 
$$\mathcal Z(P)=\Big\{z'\in B_r\,:\, P(z')=0,\ \text{Im}\,z'=0\Big\},$$
is discrete or coincides with $B_r\cap\{y'=0\}$. Now, \cref{thm:ass2} \ref{item:aa} follows since
$$
P(z')=0\qquad\Leftrightarrow
\qquad
\begin{cases}\partial_xu(x,y)=0,\\
\partial_yu(x,y)=1\,,
\end{cases}
$$
that is, every branch point $(x,y)\in\mathcal S_1(u)$ corresponds to a zero $z'$ of $P$.\smallskip

We next prove \cref{thm:ass2} \ref{item:bb}. Let $z_0=0$ be an isolated point of $\mathcal S_1(u)$ and $z_0'=0$ be the corresponding point in $\mathcal Z(P)$. Since zero is an isolated point of $\mathcal Z(P)$ and since 
$$\text{Re}\,P(z')\cdot\text{Im}\,P(z')=0\quad\text{on}\quad \{\text{Im}\,z'=0\},$$ 
we have the following three possibilities in a neighborhood of zero:
\begin{enumerate}[\rm(1)]
\item $\text{Re}\,P(z')\equiv0$ on $\{y'=0\}$, and $\text{Im}\,P(z')\neq0$ on $\{y'=0\}\setminus\{x'=0\}$;
\item $\text{Im}\,P(z')\equiv0$ on $\{y'=0\}$, and $\text{Re}\,P(z')\neq0$ on $\{y'=0\}\setminus\{x'=0\}$;
\item up to changing the direction of the real axis $\{y'=0\}$ we have 
$$\begin{cases}
\text{Re}\,P(z')\equiv0\quad\text{and}\quad\text{Im}\,P(z')\neq0\quad\text{on}\quad\{y'=0\}\cap\{x'>0\};\\ 
\text{Re}\,P(z')\neq0\quad\text{and}\quad\text{Im}\,P(z')\equiv0\quad\text{on}\quad\{y'=0\}\cap\{x'<0\}.
\end{cases}$$
\end{enumerate}	
We will show that each of these cases corresponds to one of the points \ref{item:bb1}, \ref{item:bb2} and \ref{item:bb3} of \cref{thm:ass2}. We first suppose that (3) holds. Then $P$ solves the problem
\begin{equation*}
\begin{cases}
\partial_{\bar z'}P=0\quad\text{in}\quad B_r\cap\{y'>0\},\\
\text{Re}\, P= 0\quad\text{on}\quad B_r'\cap\{x'\geq 0\},\\
\text{Im}\, P=0\quad\text{on}\quad B_r'\cap\{x'<0\}.
\end{cases}
\end{equation*}
We next notice that
$$
\partial_{x'}v-i\partial_{y'} v=Q=\frac{1+iP}{P+i}=\frac{2 {\rm Re}(P)}{|P+i|^2}-i \frac{1-|P|^2}{|P+i|^2}\,.
$$
so that
$$
\partial_{x'} v= \frac{2\,{\rm Re}(P)}{|P+i|^2}
\qquad \text{and}\qquad
\partial_{y'}v=\frac{1-|P|^2}{|P+i|^2}\,.
$$
In particular, since the function $\eta$ is increasing and $\eta(0)=0$, we get
$$
\partial_{x'}v\big(\eta(x),0\big)\equiv 0 \quad \text{for}\quad x\ge0 \,.
$$
Integrating this identity and taking into account that $v\big(\eta(0),0\big)=v(0,0)=0$, we obtain
$$
f(x)=v\big(\eta(x),0\big)=\int_{0}^x \partial_{x'}v\big(\eta(t),0\big)\, \eta'(t)\, dt=0 \quad \text{for}\quad x\ge0 \,.
$$
Conversely, assume that $x<0$ and let $x'=\eta(x)<0$. Then, ${\rm Im}(P(x'))= 0$ and
$$
\partial_{x'} v(x',0)= \frac{2P(x')}{1+P^2(x')}
\qquad \text{and}\qquad
\partial_{y'}v(x',0)=\frac{1-P^2(x')}{1+P^2(x')} \qquad \text{for}\qquad z'=x'<0\,.
$$ 
In particular, from \eqref{eq:etauu} it follows that 
$$
\begin{cases}
\displaystyle\eta'(x)=\frac{1+P^2(\eta(x))}{1-P^2(\eta(x))} &\text{ if }x<0\\
\eta(0)=0\,, 
\end{cases}
$$
which implies, by Cauchy-Kovalevskaya theorem, that $\eta\colon (-\rho,0] \to \R$ is an analytic function, with $\eta'(0)=1$, since $P(0)=0$.
Since for $x<0$ we have
\begin{equation}\label{e:eta-f}
\eta'(x)=\sqrt{1+f'(x)^2} \qquad \Rightarrow \quad f'(x)=\sqrt{\eta'(x)^2-1},
\end{equation}
we get that $f':(-\rho,0]\to\R$ is of the form 
$$f'(x)=x^{\sfrac k2} \psi(x),$$
for some $k\ge 1$ and some analytic function $\psi:(-\rho,0]\to\R$ with $\psi(0)>0$. It follows that there is an analytic function $\phi$, such that $\phi(0)>0$ and
\begin{align*}
f(x)&=
0\quad\text{if}\quad x\geq 0\qquad\text{and}\qquad
f(x)=x^{\frac{k+2}2} \phi(x)\quad\text{if}\quad x<0\,.
\end{align*}
Suppose now that (2) holds. Then $\text{Im}\, P\equiv 0$ on the real axis $\{y'=0\}$ and so, $P$ (not only $P^2$) is an holomorphic function. As a consequence, also $Q$ is holomorphic. Thus, $\partial_{y'}v(x',0)$ is analytic. Since, $\eta:(-\rho,\rho)\to\R$ solves the equation 
$$\displaystyle\eta'(x)=\frac{1}{\partial_{y'}v(\eta(x),0)}\ ,\qquad 
\eta(0)=0\,,$$
we get that $\eta$ is analytic and, by \eqref{e:eta-f}, so is $f$. This gives \ref{item:bb2}. \smallskip

\noindent Finally, we suppose that (1) holds. Since $\text{Im}\, P\neq 0$ on $\{y'=0\}\setminus\{0\}$, we get that the contact set $\mathcal C_v$ contains a neighborhood of zero. As a consequence also the contact set $\mathcal C_1(u)$ contains a neighborhood of zero (see \cref{oss:contact-set}), from which we obtain \ref{item:bb1}.\qed

\subsection{Proof of \cref{thm:example-one-phase}}Finally we come to the proof of \cref{thm:example-one-phase}, which is obtained by reversing the construction from the previous subsection. 

\begin{proof}[Proof of \cref{thm:example-one-phase}] For any $k$ of the form $k= 2n-\frac32$ with $n\in \N_{\geq1}$, we define
	$$P(z)=(iz)^k=\rho^k\Big(-\sin (k\theta)+i\cos (k\theta)\Big).$$ 
In particular, setting $\mathcal C_P:=\{ (x,0)\in \R^2\,:\, x\geq0 \} $ we have
\begin{equation*}
\begin{cases}
\partial_{\bar z}P=0\quad\text{in}\quad \{y>0\},\\
\text{Re}\, P= 0\quad\text{and}\quad\text{Im}\, P>0\quad\text{on}\quad \{x>0\}\\
\text{Re}\, P<0\quad\text{and}\quad\text{Im}\, P=0\quad\text{on}\quad \{x<0\}\,.
\end{cases}
\end{equation*}
Then we consider a radius $r\in(0,1)$ and the function $Q\colon B_r\cap \{y\geq 0\} \to \C$ 
$$
Q=\frac{1+iP}{P+i}=\frac{2\,{\rm Re}(P)}{|P+i|^2}-i \frac{1-|P|^2}{|P+i|^2}\,.
$$
Notice that $Q$ is still conformal in $B_r\cap \{y> 0\}$ and that we have 
\begin{equation*}
\begin{cases}
\partial_{\bar z}Q=0\quad\text{in}\quad \{y>0\},\\
\text{Re}\, Q= 0\ ,\quad \text{Im}\,Q\in(-1,0)\quad\text{and}\quad |Q|<1\quad\text{on}\quad \{x>0\}\,,\\
\text{Re}\, Q<0,\quad \text{Im}\,Q\in(-1,0)\quad\text{and}\quad|Q|=1\quad\text{on}\quad \{x<0\}\,.
\end{cases}
\end{equation*}
Since $B_r\cap\{y> 0\}$ is simply connected, there is a function $v\colon B_r\cap\{y\ge0\} \to \R$ such that
$$
\partial_{z}v=\partial_xv-i\partial_yv=Q \quad \text{in}\quad B_r\cap\{y>0\}\,.
$$
Precisely, for every $z=x+iy$ in $B_r\cap\{y\ge0\}$, $v$ is given by the formula 
$$v(z)=v(x,y)=\int_0^1\Big(x\,\text{Re}\,Q(tz)-y\,\text{Im}\,Q(tz)\Big)\,dt.$$
Thus, $v$ is a solution to the problem
\begin{equation*}
\begin{cases}
\Delta v=0\quad\text{in}\quad B_r\cap \{y>0\},\\
v= 0\quad\text{and}\quad |\nabla v|<1\quad\text{on}\quad B_r\cap\{x>0\}\,,\\
v>0\quad\text{and}\quad|\nabla v|=1\quad\text{on}\quad B_r\cap \{x<0\}\,.
\end{cases}
\end{equation*}
Moreover, we notice that $v(0,0)=0$ and $\partial_yv(0,0)=1$. Thus, by choosing $r>0$ small enough, we may suppose that $v>0$ in $B_r\cap \{y>0\}$.
We next consider the harmonic conjugate $V\colon B_r\cap\{y>0\} \to \R$ of $v$ and the inverse hodograph transform 
$$
S\colon B_r\cap\{y\ge 0\} \to \R^2\ ,\quad S(x,y):=\big(V(x,y), v(x,y)\big)\,.
$$
Tracing backwards the argument from \cref{ss:hodotrans}, we have that when $r$ is small enough, $S$ is a diffeomorphism; we can then consider its inverse 
$$T:S\big(B_r\cap\{y\ge 0\}\big)\to B_r\cap\{y\ge 0\}\ ,\qquad T(x',y')=\big(U(x',y'), u(x',y')\big),$$ 
where we notice that the positivity set $\Omega_u=\{u>0\}$ of the second component $u$ of $T$ is precisely $S\big(B_r\cap\{y>0\}\big)$ and that, since $v\ge 0$, $\Omega_u=S\big(B_r\cap\{y>0\}\big)$ is contained in the upper half-plane $\{y'>0\}$. Now, reasoning as in \cref{lem:con1ph} (see \eqref{e:gradv}), we get that 
$$|\nabla u(x',y')|\,|\nabla v(x,y)|=1,$$
and that, in a small ball $B_\rho$, $u$ is a solution to the problem
\begin{align}
	\Delta u=0\quad\text{in}\quad \Omega_u\cap B_\rho,\\
	u=0\quad\text{on}\quad B_\rho\cap\{y'=0\},\\
	|\nabla u|=1\quad\text{on}\quad \partial\Omega_u\cap\{y'>0\},\\
	|\nabla u|\ge 1\quad\text{on}\quad \partial\Omega_u\cap\{y'=0\},
	\end{align}
where $\partial\Omega_u\cap\{y'=0\}=\{x'\ge 0\}\cap\{y'=0\}$ and $|\nabla u|\ge 1$ {on} $\{x'\ge 0\}\cap\{y'=0\}$.	
We now define the function $f$ describing the boundary $\partial \Omega_u$ (see \eqref{eq:fbgraph}) and the function $\eta(x)=U(x,f(x))$ to be as in the proof of \cref{thm:ass2}. Then, $\eta$ is a solution to   
$$
\begin{cases}
\displaystyle\eta'(x)=\frac{1+P^2(\eta(x))}{1-P^2(\eta(x))} &\text{ if }x<0\\
\eta(0)=0\,, 
\end{cases}
$$
and so, it is analytic since $P^2(z)=iz^{4n-3}$ with $n\in\N$. Finally, since $\eta(x)=x+o(x),$ we can write the function $\eta$ as 
$$|\eta(x)|^{\sfrac12}=|x|^{\sfrac12}\psi(x)\quad\text{ for }x\le0,$$
where $\psi$ is analytic and $\psi(0)=1$. Thus, we get the precise form of $f$ by the formula
$$
f(x)=v(\eta(x),0)=\begin{cases}
\displaystyle{\int_0^{x} \frac{-|\eta(t)|^{2n-\sfrac12} }{|\eta(t)|^{4n-3}+1} dt} &\text{if }x<0,\smallskip\\
\quad 0&\text{if }x\geq 0\,,
\end{cases}
$$	
and we notice that $f(x)=|x|^{2n-\sfrac12}\big(1+o(1)\big)$ for $x<0$. This concludes the proof.	
\end{proof}

\section{The symmetric two-phase problem and some remarks}\label{ss:2phase}
Let $0=z_0\in \mathcal S$ and let $f_\pm$ be as in \eqref{eq:2phgraph}. We define 
	$$\Omega_\rho^\pm=\Big\{(x,y)\in\mathcal I_\rho\times \mathcal I_\rho\ :\ f_\pm(x)>y\Big\},$$
	and 
	$$\Gamma_\rho^\pm:=\Big\{(x,y)\in\mathcal I_\rho\times \mathcal I_\rho\ :\ f_\pm(x)=y\Big\}.$$
	In what follows, we perform the hodograph transform of $u_+$ in $\Omega_\rho^+$ and in $u_-$ in $\Omega_\rho^-$.\\ In order to simplify the notation, we set
	$$i:=+\ \text{or}\ -.$$
	Let $\eta_\pm, T_\pm=(U_\pm, u_\pm), S_\pm=(V_\pm, v_\pm)$ be the functions constructed in \cref{ss:har} and \cref{ss:hodotrans} separately for $u_+$ and $u_-$. Recall that the functions $v_i$, $i=\pm$, contain all the information of the free boundaries $\Gamma_\rho^i$. Precisely, for every $x\in \mathcal I_\rho$, we have 
	$$y=f_i(x)\ \Leftrightarrow\ (x,y)\in\Gamma_\rho^i\ \Leftrightarrow\ y=v_i\big(U_i(x,y),u_i(x,y)\big)=v_i(x',0).$$
	As a consequence, we get the equation 
	$$f_i(x)=v_i(\eta_i(x),0)\quad\text{for every}\quad x\in\mathcal I_\rho.$$
	In particular, we have 
	\begin{equation}\label{e:order-two-phase}
	v_+(\eta_+(x),0)\ge v_-(\eta_-(x),0)\quad\text{for every}\quad x\in\mathcal I_\rho.
	\end{equation}

	\begin{lemma}\label{lem:2phlemma}
	There is $r>0$ such that  
		$$B_r\cap \{y'\ge 0\}\subset T_+\big(\Omega_\rho^+\cup\Gamma_\rho^+\big)\qquad\text{and}\qquad B_r\cap \{y'\le 0\}\subset T_-\big(\Omega_\rho^-\cup\Gamma_\rho^-\big)\,.$$
		The functions
		$$v_\pm:B_r\cap\{\pm y'\ge 0\}\to\R\,,$$
		are both  $C^{1,\alpha}$-regular respectively in the half-disks $B_r\cap\{\pm y'\ge 0\}$ up to the hyperplane $\{y'=0\}$, and are $C^{\infty}$ respectively in $B_r\cap\{\pm y'> 0\}$. 
		Furthermore they solve the following thin two-membrane problem
		\begin{align*}
		\Delta v_+=0&\quad\text{in}\quad B_r\cap\{y'>0\},\\
		\Delta v_-=0&\quad\text{in}\quad B_r\cap\{y'<0\},\\
		v_+\big(\eta_+(x),0\big)\ge 	v_-\big(\eta_-(x),0\big)&\quad\text{for}\quad x\in\mathcal I_\rho,\\
		|\nabla v_\pm|	(\eta_\pm(x),0)=1&\quad\text{when}\quad v_+(\eta_+(x),0)> 	v_-(\eta_-(x),0),\\
		\eta_+'(x)\,\partial_{y'}v_+(\eta_+(x),0)=		\eta_-'(x)\,\partial_{y'}v_-(\eta_-(x),0)\le1&\quad\text{when}\quad v_+(\eta_+(x),0)= 	v_-(\eta_-(x),0),
		\end{align*}
		Moreover, for every $x\in \Gamma_\rho$ we have the identities
		\begin{equation}\label{eq:etauu2}
		f'_\pm(x)= \pm\frac{\partial_{x'}v_\pm(\eta_\pm(x),0)}{\partial_{y'}v_\pm(\eta_\pm(x),0)}
		\qquad \text{and}\qquad 
		\eta_\pm'(x)=\frac{1}{\partial_{y'}v_\pm(\eta_\pm(x),0)}.
		\end{equation}
	\end{lemma}	
	\begin{proof} We reason precisely as in \cref{lem:con1ph}.
		Since 
		$$v_i\big(U_i(x,y),u_i(x,y)\big)=y\quad\text{for every}\quad (x,y)\in \Omega_\rho^i,$$
		taking the derivatives with respect to $x$ and $y$, we obtain that
		\begin{equation*}
		\begin{cases}
		\partial_{x'} v_i\big(U_i(x,y),u_i(x,y)\big)\partial_xU_i(x,y)+	\partial_{y'} v_i\big(U_i(x,y),u_i(x,y)\big)\partial_xu_i(x,y)=0,\\
		\partial_{x'} v_i\big(U_i(x,y),u_i(x,y)\big)\partial_yU_i(x,y)+	\partial_{y'} v_i\big(U_i(x,y),u_i(x,y)\big)\partial_yu_i(x,y)=1.
		\end{cases}
		\end{equation*}
		Since, $\partial_xU_i=\partial_yu_i$ and $\partial_yU_i=-\partial_xu_i$, we get
		\begin{equation*}
		\begin{cases}
		-\partial_{x'} v_i(x',y')\partial_yu_i(x,y)+	\partial_{y'} v_i(x',y')\partial_xu_i(x,y)=0,\\
		\partial_{x'} v_i(x',y')\partial_xu_i(x,y)+	\partial_{y'} v_i(x',y')\partial_yu_i(x,y)=1.
		\end{cases}
		\end{equation*}
		When $y'=0$, we can write 
		$$x'=\eta_i(x)\qquad\text{and}\qquad y=f_i(x).$$
		Thus, we have 
			\begin{equation*}
					\begin{cases}
		-\partial_{x'} v_i(\eta_i(x),0)\partial_yu_i(x,f_i(x))+	\partial_{y'} v_i(\eta_i(x),0)\partial_xu_i(x,f_i(x))=0,\\
		\partial_{x'} v_i(\eta_i(x),0)\partial_xu_i(x,f_i(x))+	\partial_{y'} v_i(\eta_i(x),0)\partial_yu_i(x,f_i(x))=1,
				\end{cases}
		\end{equation*}
		which we will simply write as
		\begin{equation}\label{e:main-system-u-v}
		\begin{cases}
		\begin{array}{ll}
		-\partial_{x'} v_i\,\partial_yu_i+	\partial_{y'} v_i\,\partial_xu_i=0,\\
		\ \, \,\partial_{x'} v_i\,\partial_xu_i+	\partial_{y'} v_i\,\partial_yu_i=1,
		\end{array}
		\end{cases}
		\end{equation}
		and we remember that all the derivatives of $v$ are computed in $(\eta_i(x),0)$, while all the derivatives of $u$ are calculated in $(x,f_i(x))$.
		We next consider two cases: \smallskip
		
		\noindent {\bf Case 1.} $v_+(\eta_+(x),0)=v_-(\eta_-(x),0)$. We set 
		$$f(x):=f_+(x)=f_-(x)\qquad\text{and}\qquad f'(x):=f_+'(x)=f_-'(x),$$ 
		and we notice that we have the system 
		\begin{align}
		\partial_xu_++f'(x)\partial_yu_+=0=
		\partial_xu_-+f'(x)\partial_yu_-\label{e:system1-derivatives-u:1}\\
				-f'(x)\partial_xu_++\partial_yu_+=-f'(x)\partial_xu_-+\partial_yu_-	\label{e:system1-derivatives-u:2}\\
	-f'(x)\partial_xu_\pm+\partial_yu_\pm\ge \big(1+(f'(x))^2\big)^{\sfrac12}.\label{e:system1-derivatives-u:3}		
\end{align}
		where again all the partial derivatives of $u_+$ and $u_-$ are computed in 	$(x,f(x))$.\\	 Now, using \eqref{e:system1-derivatives-u:1} in \eqref{e:system1-derivatives-u:2} and \eqref{e:system1-derivatives-u:3}, we get 
		\begin{equation}\label{e:system1-derivatives-u:22}
		\partial_yu_+=\partial_yu_-
		\end{equation}	
			\begin{equation}\label{e:system1-derivatives-u:33}		
			\sqrt{1+(f'(x))^2}\ \partial_yu_\pm	\ge 1.
		\end{equation}
	On the other hand, using \eqref{e:system1-derivatives-u:1} in the system \eqref{e:main-system-u-v}, it becomes
		\begin{equation}\label{e:main-system-u-v-2}
	\begin{cases}
	\begin{array}{ll}
	\big(\partial_{x'} v_i+	\partial_{y'} v_i\,f'(x)\big)\,\partial_yu_i=0,\\
	\big(-f'(x)\,\partial_{x'} v_i+	\partial_{y'} v_i\big)\,\partial_yu_i=1,
	\end{array}
	\end{cases}
	\end{equation}
	so we get 
	$$	\big(1+f'(x)^2\big)\, \partial_{y'} v_\pm\,\partial_yu_\pm=1,$$
	which gives that  
	$$\partial_{y'} v_+=\partial_{y'} v_-\ ,\qquad \partial_{x'} v_+=\partial_{x'} v_- \qquad\text{and}\qquad \sqrt{1+(f'(x))^2}\ \partial_{y'}v_\pm	\le 1\,, $$
all the derivatives of $v_\pm$ being calculated in $(\eta_\pm(x),0)$.
		\medskip
		
		\noindent {\bf Case 2.}  $v_+(\eta_+(x),0)>v_-(\eta_-(x),0)$. In this case the two free boundaries separate, that is $f_+>f_-$ in a neighborhood of $x$. Then, for each $i=\pm$, we can proceed as in the proof of \eqref{e:system00-w:boundary-condition-w>0} in \cref{lem:con1ph}. \smallskip

		Finally, we notice that \eqref{eq:etauu2} follows by taking the reflection $\bar{u}(x,y):=-u_-(x,-y)$ and applying the identities from \eqref{eq:etauu} to  $u_+$ and $\bar u$.
	\end{proof}	

When $u$ is a symmetric solution to the two-phase problem, we have the following
\begin{corollary}\label{c:symmetric}
Let $u$ be a symmetric solution to the two-phase problem, then, up to taking a smaller radius $r>0$, the functions $v_\pm$ constructed in \cref{lem:2phlemma} satisfy
\begin{align*}
		\Delta v_+=0&\quad\text{in}\quad B_r\cap\{y'>0\},\\
		\Delta v_-=0&\quad\text{in}\quad B_r\cap\{y'<0\},\\
		|\nabla v_\pm|	(x',0)=1&\quad\text{when}\quad x'\in B'_r\setminus\mathcal C_v\\
		|\nabla v_+|(x',0)=|\nabla v_-|(x',0)\le1&\quad\text{when}\quad  B'_r\cap\mathcal C_v\,,
\end{align*}
where we denote by $\mathcal C_v$ the contact set 
\begin{equation}\label{e:def-Cv-two-phase}
\mathcal C_v:=\big\{(x',0)\ :\ x'=\eta(x),\ x\in \mathcal I_\rho,\ f_+(x)=f_-(x)\big\}\,.
\end{equation}
\end{corollary}

\begin{proof} By definition 
	$$
	\eta_\pm(x)=\int_0^x |\nabla u_\pm|(t, f_\pm(t))\,\sqrt{1+|f'_\pm(t)|^2}\,dt\,.
	$$
	Let $\mathcal I_i$ be the intervals defined in \eqref{e:onephase}, then notice that
	\begin{itemize}
		\item if $t\in \mathcal I_i$, then $|\nabla u_\pm|(t, f_\pm(t))=1$;\smallskip
		\item if $t\in (-\rho,\rho)\setminus \left(\bigcup_i \mathcal I_i\right)$, then $f_+(t)=f_-(t)$ and $|\nabla u_+|(t, f(t))=|\nabla u_-|(t, f(t))$.
	\end{itemize}
	In particular the first bullet implies that 
$$
\eta_+(\mathcal I_i)=\eta_-(\mathcal I_i)\qquad \forall i\,, 
$$ 
which combined with the second bullet implies that  
$$
\eta_+\Big(\big\{x\in (-\rho,\rho)\,:\,f_+(x)>f_-(x)\big\}\Big)=\eta_-\Big(\big\{x\in  (-\rho,\rho)\,:\,f_+(x)>f_-(x)\big\}\Big)\,,
$$
from which the conclusion follows from the previous lemma. 
\end{proof}

\begin{oss}
Notice that, in the above proof, we are not claiming that $\eta_+\equiv \eta_-$, but only that branch points are sent in branch points.
\end{oss}

\subsection{Proof of \cref{thm:ass4} (a)}\label{ss:two-phase:a} Let $v_\pm$ be the functions from \cref{c:symmetric} and let
\begin{equation}\label{e:def:Q-two-phase}
Q_\pm:=\partial_{x'}v_\pm-i\partial_{y'}v_\pm
\end{equation}
As in the proof of \cref{thm:ass2}, we have that $Q$ is a solution to 
\begin{equation}\label{e:equations-Q-two-phase}
\begin{cases}
\partial_{\bar z}Q_\pm=0\quad\text{in}\quad B_r\cap\{\pm y'>0\}\,,\\
|Q_\pm|= 1\quad\text{on}\quad B_r\cap\{y'=0\}\setminus\mathcal C_v\,,\\
Q_+=Q_-\quad\text{on}\quad B_r\cap\{y'=0\}\cap\mathcal C_v\,.
\end{cases}
\end{equation}
We then define 
\begin{equation}\label{e:def:P-two-phase}
P_\pm=-i\frac{Q_\pm+i}{Q_\pm-i}=-i\frac{(Q_\pm+i)(\bar Q_\pm+i)}{|Q_\pm-i|^2}=\frac{2\,\text{Re}\,Q_\pm}{|Q_\pm-i|^2}-i\,\frac{|Q_\pm|^2-1}{|Q_\pm+i|^2}\,,
\end{equation}
and we notice that 
\begin{equation*}
\begin{cases}
\partial_{\bar z}P_\pm=0\quad\text{in}\quad B_r\cap\{\pm y'>0\}\,,\\
P_+=P_- \quad\text{on}\quad B_r\cap\{y'=0\}\cap\mathcal C_v\,,\\
\text{Im}\, P_\pm=0\quad\text{on}\quad B_r\cap\{y'=0\}\setminus\mathcal C_v\,.
\end{cases}
\end{equation*}
We now consider the reflection 
$$
P'\colon B_r\cap \{y'\geq 0\} \to \C\ ,\qquad P'(z):=\overline P_-(\bar z)\,,
$$
so that the functions $P_+$ and $P'$ are both defined on the same domain and we can take
\begin{equation}\label{e:MandD}
M(z):= \frac{P_+(z)+P'(z)}2\qquad \text{and}\qquad D(z):= \frac{P_+(z)-P'(z)}2\,,
\end{equation}
which satisfy the equations
\begin{equation}\label{e:two-phase-M-full}
\begin{cases}
\partial_{\bar z}M=0\quad\text{in}\quad B_r\cap\{y'>0\},\\
\text{Im}\, M= 0\quad\text{on}\quad B_r\cap\{y'=0\}.
\end{cases}
\end{equation}
and
\begin{equation*}
\begin{cases}
\partial_{\bar z}D=0\quad\text{in}\quad B_r\cap\{y'>0\},\\
{\rm Re} D=0 \quad\text{on}\quad B_r\cap\{y'=0\}\cap\mathcal C_v,\\
\text{Im}\, D=0\quad\text{on}\quad B_r\cap\{y'=0\}\setminus\mathcal C_v,
\end{cases}
\end{equation*}
Reasoning as in the proof of \cref{thm:ass2}, $D^2$ we get that $\text{Im}(D^2)=2\text{Re} D\,\text{Im} D=0$ on $\{y'=0\}$ so that $D^2$ can be extended to a conformal map on to the whole of $B_r$, so the set
$$\{D=0\}\cap B_r\cap\{y'=0\},$$
is either discrete or coincides with $B_r\cap\{y'=0\}$. This proves \cref{thm:ass4} \ref{item:aaa} since at every $z'$ on the real line $\{y'=0\}$ we have
$$
D(z')=0\quad\Leftrightarrow
\quad
\begin{cases}P^+=P^-\\
\text{Im} P_\pm=0
\end{cases}\quad\Leftrightarrow
\quad
\begin{cases}Q^+=Q^-\\
|Q_\pm|=1
\end{cases}
\quad\Leftrightarrow
\quad
\begin{cases}\nabla u_+=\nabla u_-\\
|\nabla u_\pm|=1\,,
\end{cases}
$$
that is every branch point of $u$ corresponds to a zero of $D$. \qed

\subsection{Proof of \cref{thm:ass4} (b) and \cref{c:symmetric}} 
\begin{oss}
We notice that in this part of \cref{thm:ass4} we do not assume any symmetry of the solutions, but only that the branch points are isolated. 
\end{oss}	
\noindent Let $z_0\in \mathcal S_{2}(u_+,u_-)$ be an isolated point of $\mathcal S_{2}(u_+,u_-)$. If $z_0$ is in the interior of the contact set $\mathcal C_2(u_+,u_-)$, then \ref{item:bb2} is immediate as the function 
$u=u_+-u_-$ is harmonic in a neighborhood of $z_0$. Suppose then that $z_0$ is a branch point: $z_0\in\mathcal B_2(u_+,u_-)$; moreover, since $\mathcal B_2\subset\mathcal S_2$, we have that $z_0$ is isolated in the set of branch points $\mathcal B_2(u_+,u_-)$. This means that in order to complete the proof of \cref{thm:ass4} \ref{item:bbb} we only need to prove \cref{c:symmetric}. We set $z_0=0$ and we consider the following two cases:\smallskip

\noindent {\bf Case 1.} $0$ is isolated also as point of the contact set $\mathcal C_2(u_+,u_-)$, that is $B_r\cap \mathcal C_2(u_+,u_-)=\{0\}$ for some radius $r>0$. In this case, on the free boundaries $\partial\Omega_u^\pm$ we have that $|\nabla u_\pm|=1$ and so, \cref{c:symmetric} \ref{item:bbb1} follows as in the proof of \cref{thm:ass2} \ref{item:bb1}.\smallskip

\noindent{\bf Case 2.} $0$ is not isolated in the set $\mathcal C_2(u_+,u_-)$. Then, since there are no other branch points in a neighborhood of $0$, we can assume that: 
$$f_+(x)=f_-(x)\quad\text{when}\quad x\ge 0\qquad\text{and}\qquad f_+(x)>f_-(x)\quad\text{when}\quad x<0.$$
As above, we define $\eta_\pm$ as 
\begin{equation}\label{e:eta-two-phase}
\eta_\pm(x)=\int_0^x|\nabla u_\pm|(t,f_\pm(t))|\sqrt{1+(f_\pm'(t))^2}\,dt\,,
\end{equation}
while $v_\pm$ are the hodograph transforms of $u_\pm$, for which we recall the identities
$$f_\pm(x)=v_\pm(\eta_\pm(x),0)\quad\text{and}\quad |\nabla v_\pm|(\eta_\pm(x),0)=\frac{1}{|\nabla u|(x,f_\pm(x))}\,.$$
for every $x$ in a neighborhood of zero. Then, since $\eta_+(x)=\eta_-(x)$ for $x\ge 0$, we get that: 
$$\begin{cases}
v_+(x',0)=v_-(x',0)\quad\text{and}\quad \nabla v_+(x',0)=\nabla v_-(x',0)&\quad\text{when}\quad x'\ge 0,\\
|\nabla v_+|(x',0)=|\nabla v_-|(x',0)&\quad\text{when}\quad x'<0.
\end{cases}$$
\begin{oss}
Notice that when $x<0$ we cannot say if $\eta_+(x)=\eta_-(x)$. In particular, we cannot say if $v_+(x',0)\ge v_-(x',0)$ when $x'<0$ and so, we don't know if $\{x'\ge 0\}$ is the contact set $\{x'\ :\ v_+(x',0)=v_-(x',0)\}$. 
\end{oss}	
We next consider the functions $Q_\pm$ and $P_\pm$ given by \eqref{e:def:Q-two-phase} and \eqref{e:def:P-two-phase}, and the functions $D$ and $M$ defined in \eqref{e:MandD}. Then, in a neighborhood $(-r,r)\times[0,r)$ of zero, the difference  $D$ satisfies
 \begin{equation}\label{e:to-phase-D-half}
 \begin{cases}
 \partial_{\bar z}D=0\quad\text{in}\quad (-r,r)\times (0,r),\\
 {\rm Re} D=0 \quad\text{on}\quad (0,r)\times\{0\}\\
 \text{Im}\, D=0\quad\text{on}\quad  (-r,0)\times\{0\}\,.
 \end{cases}
 \end{equation}
Recall that by the definitions of $M$, $D$ and $P'$, we have 
$$P_+(z)=M(z)+D(z)
\qquad\text{and}\qquad
P_-(z)=\overline{M(\bar z)}-\,\overline{D(\bar z)}
$$
and moreover
$$
\partial_{x'} v_\pm= \text{Re}(Q_\pm)=\frac{2\,{\rm Re}(P_\pm) }{|P_\pm+i|^2}
\qquad \text{and}\qquad
\partial_{y'} v_\pm=- \text{Im}(Q_\pm) =\frac{1-|P_\pm|^2}{|P_\pm+i|^2}\,.
$$
We set $g_\pm(x'):=\eta_\pm^{-1}(x')$ and $\tilde f_\pm(x'):=f_\pm(g_\pm(x'))$. Since, 
$$f_\pm(x)=v_\pm(\eta_\pm(x),0)\qquad\text{and}\qquad\displaystyle \eta_\pm'(x)=\frac{1}{\partial_{y'}v_\pm\big(\eta_\pm(x),0\big)},$$ 
we get that
 $$
 \tilde f_\pm(x')=v_\pm(x',0)
 \qquad \text{and}\qquad
 g'_\pm(x')=\partial_{y'}v_\pm(x',0)\,.
 $$
In particular,
 $$
  \tilde f_\pm(x')=\int_0^{x'} \partial_{x'} v_\pm(t,0)\,dt=\int_{0}^{x'} \frac{2\,{\rm Re}(P_\pm(t))}{|P_\pm(t)+i|^2}\,dt
 $$
 and 
 $$
 g_\pm(x')=\int_0^{x'} \partial_{y'} v_\pm(t,0)\,dt=\int_0^{x'}\frac{1-|P_\pm(t)|^2}{|P_\pm(t)+i|^2}\,dt\,.
 $$
 Now, by \eqref{e:to-phase-D-half} and \eqref{e:two-phase-M-full}, we have that 
 $$M=\text{Re}\,M\quad\text{and}\quad D=i\,\text{Im}\,D\quad\text{on}\quad [0,r)\times \{0\}\,,$$
 which gives that on $[0,r)\times \{0\}$, $P_+=P_-$, precisely:
 $${\rm Re}(P_+)={\rm Re}(P_-)=M\quad\text{and}\quad {\rm Im}(P_+)={\rm Im}(P_-)=\text{Im}\,D=-iD.$$
This implies that
 $$
 \tilde f_{\pm}(x')=\int_0^{x'} \frac{2\,M(t)}{M^2(t)+\big(1+ \text{Im}\,D(t)\big)^2} \,dt\,,
 $$
 so that $\tilde f_{+}\equiv \tilde f_-$ on $\{x'\ge 0\}$. Similarly,
 $$
 g_\pm(x')=\int_0^{x'}\frac{1-M^2(t)-\big(\text{Im}\,D(t)\big)^2}{M^2(t)+\big(1 +\text{Im}\,D(t)\big)^2}\,dt\,,
 $$
 which again implies that $g_+\equiv g_-$. Combining these two identities, we get that 
 $$f_+\equiv f_-\quad\text{on}\quad\{x'\ge 0\}.$$
 Using again \eqref{e:to-phase-D-half} and \eqref{e:two-phase-M-full}, this time for  $x'\leq0$, we get that 
  $$M=\text{Re}\,M\quad\text{and}\quad D=\text{Re}\,D\quad\text{on}\quad(-r,0)\times\{0\}\,,$$
  which implies that $P_\pm$ are both real and 
  $$P_+=M+D\quad\text{and}\quad P_-=M-D\quad\text{on}\quad (-r,0)\times\{0\}\,.$$
As above, we compute
  $$
  \tilde f_{\pm}(x')=2\int_0^{x'} \frac{M(t)\pm D(t)}{1+(M(t)\pm D(t))^2} \,dt\qquad\text{and}\qquad
  g_\pm(x')=\int_0^{x'}\frac{1-(M(t)\pm D(t))^2}{1+(M(t)\pm D(t))^2}\,dt\,.
 $$
We now define
$$
\Psi(x'):=\frac{\tilde f_+(x')-\tilde f_-(x')}2= 2\int_0^{x'} D(t) \, \frac{1+D^2-M^2}{(1+M^2+D^2)^2-4 D^2 M^2} \,dt
$$
 and 
 $$
 \Phi(x'):=\frac{\tilde f_+(x')+\tilde f_-(x')}2= 2\int_0^{x'} M(t) \, \frac{1+M^2-D^2}{(1+M^2+D^2)^2-4 D^2 M^2} \,dt
 $$
 and we notice that:
 \begin{itemize}
 \item 	$\Phi$ is an analytic function of the form $\Phi(x')=O(x'^2)$;
 \item $\Psi$ is of the form $\Psi(x')=(x')^{\sfrac32}\Theta(x')$, where $\Theta$ is an analytic function.
 \end{itemize}	
Also, let
 \begin{align*}
 \psi:=\frac{g_+(x')-g_-(x')}2
 &=\int_0^{x'}\frac{-4D(t)M(t)}{(M^2+D^2+1)^2-4 M^2 D^2}\,dt\,,
 \end{align*}
 and
\begin{align*}
\phi:= \frac{g_+(x')+g_-(x')}2
 	&=\int_0^{x'}\frac{1-\big(M^2-D^2\big)^2}{(M^2+D^2+1)^2-4 M^2 D^2}\,dt\,,
 \end{align*}
where, as above, 
\begin{itemize}
	\item 	$\phi$ is an analytic function of the form $\phi(x')=x'+o(x')$;
	\item $\psi$ is of the form $\psi(x')=(x')^{\sfrac52}\theta(x')$, where $\theta$ is an analytic function.
\end{itemize}
Therefore we have
 $$
 \begin{cases}
 f_+\big(\phi(x')+\psi(x')\big)-f_-\big(\phi(x')-\psi(x')\big)=2\Psi(x'),\\
  f_+\big(\phi(x')+\psi(x')\big)+f_-\big(\phi(x')-\psi(x')\big)=2\Phi(x'),
 \end{cases}
 $$
and
 $$f_+\big(\phi(x')+\psi(x')\big)=\Phi(x')+\Psi(x')\qquad\text{and}\qquad f_-\big(\phi(x')-\psi(x')\big)=\Phi(x')-\Psi(x').$$
 Since $\eta_\pm$ is the inverse of $\phi\pm\psi$, we get that $\eta_\pm$ of the form
$$\eta_\pm(x)=x+x^{\sfrac52}\beta_{\pm}(x^{\sfrac12}),$$
where $\beta_\pm$ are analytic functions. Thus, 
$$f_\pm(x)=\Phi\Big(x+x^{\sfrac52}\beta_{\pm}\big(x^{\sfrac12}\big)\Big)\pm\Psi\Big(x+x^{\sfrac52}\beta_{\pm}\big(x^{\sfrac12}\big)\Big),$$
which concludes the proof of \cref{c:symmetric} and \cref{thm:ass4} \ref{item:bbb3}.\qed

\subsection{Remarks on the non-symmetric case}For non-symmetric solutions, or more generally when different weights are put on the gradients of $u_\pm$ (as in the more general Alt-Caffarelli-Friedman energy, see for instance \cite{DSV}), we cannot guarantee the validity of \cref{c:symmetric}, and so branch points of the original problem might not be sent into branch points of the thin two-membrane problem. In fact, suppose that $(x_0,f_\pm(x_0))$ and $(x_1,f_\pm(x_1))$ are two consecutive points in $\mathcal B_2(u_+,u_-)$ such that $x_0<x_1$ and 
$$\begin{cases}
f_+(x)=f_-(x)\quad\text{when}\quad x\le x_0,\\
f_+(x)>f_-(x)\quad\text{when}\quad x_0<x<x_1,\\
f_+(x)=f_-(x)\quad\text{when}\quad  x\ge x_1.\end{cases}$$ 
Suppose that $x_0=0$ and define $\eta_\pm$ as in \eqref{e:eta-two-phase}. Now, we might have that 
\begin{equation}\label{e:eta-final-remarks}
\eta_+(x_1)=\int_{0}^{x_1}\sqrt{1+(f_+'(t))^2}\,dt>\int_{0}^{x_1}\sqrt{1+(f_-'(t))^2}\,dt=\eta_-(x_1).
\end{equation}
But then, for a generic point $x'$ between $\eta_-(x_1)$ and $\eta_+(x_1)$, we get that $|\nabla v_+|(x',0)=1$, while $|\nabla v_-|(x',0)<1$, so that the equations \eqref{e:equations-Q-two-phase} for $Q_\pm$ are not satisfied. 

\smallskip

We notice that the symmetry assumption in point \ref{item:aaa} of \cref{thm:ass4} is precisely what prevents \eqref{e:eta-final-remarks} from happening. In particular, this assumption is fulfilled when 
\begin{equation}\label{e:final-remarks-f}
f_+(x)+f_-(x)\equiv 0\quad\text{on}\quad B_1'.
\end{equation}
 We also notice that \eqref{e:final-remarks-f} is equivalent to assuming that $\eta_+\equiv\eta_-$.

\begin{lemma}
	Suppose that $\eta_+\equiv\eta_-$ on $(-1,1)$, then $u_\pm\colon B_1^\pm\cup B_1'\to \R$ and moreover
	$$
	u_-(x,y)=-u_+(x,-y)\quad \text{and}\quad f_+(x)+f_-(x)=0\quad\text{for every}\quad x\in (-1,1).$$
\end{lemma}

\begin{proof}
	Since $\eta'_+\equiv \eta_-'$, \eqref{eq:etauu2} implies that $\partial_{y'}v_+(\eta_+(x),0)=\partial_{y'}v_-(\eta_-(x),0)$. In particular, 
	\begin{itemize}
		\item if $f_+(x)>f_-(x)$, then $|\nabla v_\pm(\eta(x),0)|=1$ and so $\partial_x v_+(\eta_+(x),0)=\partial_x v_-(\eta_-(x),0)$;
		\item if $f_+(x)=f_-(x)$, then $\partial_x v_+(\eta_+(x),0)=\partial_x v_-(\eta_-(x),0)$.
	\end{itemize}
	In conclusion we have that 
	$$
	\nabla v_+ (\eta_+(x),0)=\nabla v_- (\eta_-(x),0)\,,
	$$
	which using again \eqref{eq:etauu2} implies that $f'_+(x)\equiv -f'_-(x)$. Since $f_\pm(0)=0$, integrating we get
	$$
	f_+(x)+f_-(x)=\int_0^x (f'_+(t)+f'_-(t))\,dt=0\,.
	$$
Finally, $u_-(x,y)+u_+(x,-y)$ is a harmonic function in $\Omega_u^-$ which vanishes together with its gradient on $\partial\Omega_u^-$. This implies that $u_-(x,y)+u_+(x,-y)=0$ for every $(x,y)\in\Omega_u^-$.
\end{proof}

\end{document}